\documentclass[11pt,reqno]{amsart}
\usepackage{amsmath}
\usepackage{amsthm}
\usepackage{amssymb}
\usepackage{graphics}
\usepackage{bm}
\usepackage{setspace}
\usepackage{fullpage}

\usepackage{amsfonts}
\usepackage{latexsym}
\usepackage{mathdots}
\usepackage{mathrsfs}
\usepackage{enumitem}
\usepackage{eucal}
\usepackage{url}
\usepackage{color}

\newcommand{\be}{\begin{equation}}
\newcommand{\ee}{\end{equation}}
\newcommand{\ba}{\begin{eqnarray}}
\newcommand{\ea}{\end{eqnarray}}
\newcommand{\bi}{\begin{itemize}}
\newcommand{\ei}{\end{itemize}}
\newcommand{\bn}{\begin{enumerate}}
\newcommand{\en}{\end{enumerate}}
\newcommand{\bbm}{\begin{bmatrix}}
\newcommand{\ebm}{\end{bmatrix}}
\newcommand{\bp}{\begin{proof}}
\newcommand{\ep}{\end{proof}}
\newcommand{\nn}{\nonumber}

\newcommand{\mr}{\ensuremath{\mathrm}}

\newcommand{\mbf}{\ensuremath{\mathbf}}
\newcommand{\mc}{\ensuremath{\mathcal}}

\newcommand{\ov}{\ensuremath{\overline}}
\newcommand{\sm}{\ensuremath{\setminus}}
\newcommand{\wt}{\ensuremath{\widetilde}}

\newcommand{\Ga}{\ensuremath{\Gamma}}
\newcommand{\ga}{\ensuremath{\gamma}}
\newcommand{\Om}{\ensuremath{\Omega}}

\newcommand{\la}{\ensuremath{\lambda }}
\newcommand{\om}{\ensuremath{\omega}}

\def\C{\mathbb{C}}
\def\R{\mathbb{R}}
\def\D{\mathbb{D}}
\def\T{\mathbb{T}}
\def\Z{\mathbb{Z}}
\def\N{\mathbb{N}}

\renewcommand{\H}{\ensuremath{\mathcal{H} }}

\newcommand{\K}{\ensuremath{\mathcal{K} }}
\renewcommand{\L}{\ensuremath{\mathscr{L} }}
\newcommand{\F}{\ensuremath{\mathbb{F} }}

\newcommand{\intfty}{\ensuremath{\int _{-\infty} ^{\infty}} }

\newcommand{\ip}[2]{\ensuremath{\langle {#1} , {#2} \rangle}}

\newcommand{\dom}[1]{\ensuremath{\mathrm{Dom} ({#1}) }}
\renewcommand{\dim}[1]{\ensuremath{\mathrm{dim} \left( {#1} \right) }}
\newcommand{\ran}[1]{\ensuremath{\mathrm{Ran} \left( {#1} \right) }}

\renewcommand{\ker}[1]{\ensuremath{\mathrm{Ker} ({#1}) }}

\newcommand{\im}[1]{\ensuremath{\mathrm{Im} \left( {#1} \right) }}

\numberwithin{subsection}{section}

\newtheorem{thm}[subsection]{Theorem}
\newtheorem*{thm*}{Theorem}

\newtheorem{lemma}[subsection]{Lemma}
\newtheorem{prop}[subsection]{Proposition}
\newtheorem{cor}[subsection]{Corollary}

\theoremstyle{definition}
\newtheorem{defn}[subsection]{Definition}
\newtheorem{remark}[subsection]{Remark}
\newtheorem{eg}[subsection]{Example}

\title[Time-varying bandlimit]{Function spaces obeying a time-varying bandlimit}

\author{R.T.W. Martin}
\address{University of Cape Town}
\email{rtwmartin@gmail.com}

\author{A. Kempf}
\address{University of Waterloo}
\email{akempf@uwaterloo.ca}

\begin{document}

\bibliographystyle{unsrt}
\maketitle
\onehalfspace

\begin{abstract}
    Motivated by applications to signal processing and mathematical physics, recent work on the concept of time-varying bandwidth has produced a class of function spaces which generalize the Paley-Wiener spaces of bandlimited functions: any regular simple symmetric linear transformation with deficiency indices $(1,1)$ is naturally represented as multiplication by the independent variable in one of these spaces. We explicitly demonstrate the equivalence of this model for such linear transformations to several other functional models based on the theories of meromorphic model spaces of Hardy space and purely atomic Herglotz measures on the real line, respectively. This theory provides a precise notion of a time-varying or local bandwidth, and we describe how it may be applied to construct signal processing techniques that are adapted to signals obeying a time-varying bandlimit.

\vspace{5mm}   \noindent {\it Key words and phrases}: time-varying bandlimit, local bandwidth, sampling theory and signal processing, symmetric operators, self-adjoint extensions, Hardy spaces, deBranges spaces, reproducing kernel Hilbert spaces.

\vspace{3mm}
\noindent {\it 2010 Mathematics Subject Classification} ---47B32 Operators on reproducing kernel Hilbert spaces; 46E22 Hilbert spaces with reproducing kernels; 47B25 Symmetric and selfadjoint operators (unbounded);
\end{abstract}

\section{Introduction}


Information theory distinguishes between continuous and discrete forms of information; \emph{e.g.} music signals on one hand and discrete sets of symbols on the other. The crucial bridge between continuous and discrete is provided by Shannon sampling theory and its generalizations \cite{Shan,Jerri1977,Ben2012,Marks2012}. For example, a raw, continuous audio signal, $f_{raw} (t)$, is a pressure-valued function of time. In order to record all the information required to reconstruct $f_{raw}$ perfectly, it would appear to be necessary to record its amplitudes at all (uncountably many) points in time, and this is unfeasible. In signal processing, this problem is overcome by applying the fact that the average human ear is incapable of sensing frequencies above $22$kHz \cite[pg.163]{Rosen2011}. It follows that the audio signal $f_{raw}$ can be low pass filtered accordingly to obtain a signal $f(t)$ which contains no frequencies greater in magnitude than $A=22$kHz.  The filtered signal is said to be $A-$bandlimited, the positive number $A$ is called the bandlimit, and the subspace $B(A):= \mc{F} ^{-1} L^2 [-A, A] \subset L^2 (R)$ ($\mc{F}$ denotes Fourier transform) is called the Paley-Wiener space of $A$-bandlimited functions. As Shannon pointed out, this filtering has a tremendous benefit: In order to record and later reconstruct any such filtered signal, it suffices to record the signal's amplitudes or \emph{samples}
$\{ f(t_n ) \}$
at a discrete set of sample times $\{ t_n \}$ with spacing $t_{n+1}-t_n=\pi/A$, the so-called Nyquist spacing. Any such discrete set can be arranged in strictly increasing order and there is a one-parameter family of such sequences of sample points, or \emph{sampling sequences} $t_n (\vartheta ):= (n +\vartheta ) \frac{\pi}{A}; \ \vartheta \in [0,1)$.
For a fixed choice of $\vartheta$, the samples $\{ f(t_n(\vartheta)) \}$ then completely determine and represent the $A-$bandlimited function $f$.
Indeed, the celebrated Shannon sampling formula can be used to reconstruct $f$ perfectly (in theory) from these discrete values:
$$ f(t) = \sum _{n \in \Z } f\left( t_n (\vartheta) \right) \frac{\sin \left( A (t - t_n (\vartheta ) \right)}{A (t-t_n (\vartheta ) ) } ; \quad \quad
t_n (\vartheta ):= (n +\vartheta ) \frac{\pi}{A}, \ \vartheta \in [0,1). $$ This key reconstruction property is applied ubiquitously in signal processing to discretize and later reconstruct audio or video signals \cite{Jerri1977,Ben2012,Marks2012}.

In practical applications, the bandlimit $A$ is necessarily the largest frequency that occurs in the set of signals considered. The larger the value of $A$, the smaller the spacing $\pi/A$ of the sample times at which the samples need to be recorded. Even if a given signal appears to have low `bandwidth' for most of its duration, and to be a linear combination of a wide range of frequencies for only a short time interval, the samples need to be taken at a high rate for all time in order to apply the Shannon sampling formula. This is intuitively inefficient and motivates the extension of signal processing methods such as filtering, sampling and reconstruction to the setting of time-varying bandwidth. What exactly, however, is a time-varying bandlimit? The traditional notion of bandlimit is determined by the Fourier transform of the entire signal and hence is non-local, it depends on the signal's global behaviour.  This makes it difficult to make the concept of a time-varying bandlimit precise. For several approaches in the literature, see, for example, \cite{Groch2015,TVband1,TVband2,TVband3,TVband4,TVband5}.

Our definition of time-varying bandlimit (Definition \ref{TVbanddefn}) is based on the observation that, in conventional Shannon sampling theory, the constant bandlimit $A$ is inversely proportional to the constant spacing $\pi/A$ of the standard Nyquist sampling sequences.  We then identify the sample points in each of these sampling sequences $( t_n (\vartheta ) = (n+\vartheta ) \frac{\pi}{A} )$ for $\vartheta \in [0,1 )$, appearing in the Shannon sampling formula, with the simple eigenvalues of a self-adjoint operator $Z _\vartheta$. We further observe that the family $\{ Z _\vartheta | \ \vartheta \in [0, 1) \}$ is the one-parameter family of self-adjoint extensions of a single symmetric linear transformation $Z$ which acts as multiplication by the independent variable on a dense domain in $B(A)$ (and which is simple, regular and has deficiency indices $(1,1)$, we will recall these basic definitions in Subsection \ref{Notation}) \cite{Kempf2000,Kempf2004}. One can combine the spectra of these self-adjoint extensions to define a smooth, strictly increasing bijection on the real line, $t (n +\vartheta ) :=  t_n (\vartheta )$. If $\ga $ denotes the compositional inverse of $t$, we observe that
$$ \pi \ga ' (t) = A, $$ is the bandlimit. The derivative $\ga' (t)$ is then a measure of the local density of the sampling sequences $( t_n (\vartheta ) )$ near the point $t$, and it is proportional to the constant bandlimit in the case of Shannon sampling.

Crucially, the spectra of the self-adjoint extensions of such a symmetric operator, $T$, need not be equidistant. It is possible, therefore, to straightforwardly generalize Shannon sampling theory using the representation theory of regular simple symmetric linear transformations with defect indices $(1,1)$ (we will review the definitions in Subsection \ref{Notation}).  We will develop this theory to show that any such symmetric $T$ is unitarily equivalent to multiplication by the independent variable in a \emph{local bandlimit space}, $\K (T)$, a Hilbert space of functions on $\R$ with the same special reconstruction properties as the Paley-Wiener spaces, $B(A)$, of $A-$bandlimited functions. Namely, we will prove that any $f \in \K (T)$ can be reconstructed from its samples taken on any sampling sequence $(t_n (\theta ) )$, $\theta \in [0,1 )$, where the $t_n (\theta )$ are the simple eigenvalues of a self-adjoint extension, $T_\theta$, of $T$ (see Theorem \ref{tvbspace}).  The local density of the sampling sequences $(t_n (\theta) )$ will then provide a natural notion of time-varying bandlimit that recovers the classical definition in the case where $\K (T) = B(A)$ (Example \ref{PW}, and Subsection \ref{TVbandsection}).

The goal of this paper is two-fold. Our first aim is to apply the spectral theory of regular simple symmetric linear transformations, $T$, with indices $(1,1)$ to construct the local bandlimit spaces $\K (T)$ as introduced in \cite{Kempf2000,Kempf2004,Hao2011}, and to demonstrate that these spaces obey Shannon-type sampling formulas. We will further show that $T$ is unitarily equivalent to the operator of multiplication by the independent variable in $\K (T)$, and that $\K (T)$ can be embedded isometrically in measure spaces $L^2 (\R , d\la )$, for a class of positive measures $\la$ which are equivalent to Lebesgue measure (Theorem \ref{tvbspace}). We then develop equivalent representations of such $T$ as multiplication by the independent variable in (a) meromorphic model subspaces of Hardy space, and (b) $L^2$ spaces of functions square integrable with respect to purely atomic Herglotz measures on $\R$ whose atoms have no finite accumulation point \cite{Krein1944,Krein1944one,GG,dB,Martin-dB,Martin-sym,Martin-ext,AMR,GMR}. Connecting these theories will provide powerful new tools for studying local bandlimit spaces.  In particular, this will yield a precise notion of time-varying bandlimit. Our second aim is to apply these results to develop more efficient signal processing techniques that are adapted to time-varying bandwidths. Namely, we extend concepts and concrete tools related to filtering, sampling and reconstruction to the time-varying setting.

\subsection{Outline} \label{Outline}
Let $\mc{S} ^R$ denote the family of closed regular simple symmetric linear transformations with deficiency indices $(1,1)$ defined on a domain in some separable Hilbert space. If $T \in \mc{S} ^R$ is defined in $\H$, we write $T \in \mc{S} ^R (\H)$. We will recall the definition of these terms, and of self-adjoint extensions of symmetric linear transformations in the upcoming Subsection \ref{Notation}.

Section \ref{FAmodel} presents the theory of function spaces obeying a time-varying bandlimit as developed in \cite{Kempf2000,Kempf2004}. This is an abstract functional analytic approach to the representation theory of $\mc{S} ^R$. We begin with a spectral characterization of $\mc{S} ^R$ in Theorem \ref{FApic}. This theorem shows that there is a natural (and essentially bijective) correspondence between symmetric linear $T \in \mc{S} ^R$, and what we call \emph{bandlimit pairs} of real sequences $(\mbf{t}, \mbf{t}' )$; $\mbf{t} = (t_n )$, $\mbf{t' } = (t_n ') $ (Definition \ref{TVpairdef}). Given any bandlimit pair of sequences $(\mbf{t} , \mbf{t} ')$, Theorem \ref{FApic} shows that there is a symmetric $T \in \mc{S} ^R$ so that
the real sequence $\mbf{t}$ consists of the simple (multiplicity one) eigenvalues of a self-adjoint extension, $T_0$, of $T$.  Moreover the spectra of the entire (one-parameter) family of self-adjoint extensions of $T$ can be combined to construct a smooth, strictly increasing function on $\R$, the \emph{spectral function}, $t$, of $T$ (Definition \ref{specfundef}) so that $t_n = t (n)$ and, up to a fixed constant, $t' _n \simeq t' (n) >0$ (Lemma \ref{defcoeff}).

Given any initial bandlimit pair of sequences $(\mbf{t} , \mbf{t} ' )$ and corresponding symmetric $T \in \mc{S} ^R$, Proposition \ref{kernelprop} constructs a smooth positive kernel function $K ^T : \R \times \R \rightarrow \R$. By classical reproducing kernel Hilbert space (RKHS) theory (see Subsection \ref{RKHStheory}), there is a unique RKHS $\K (T) = \K (\mbf{t} , \mbf{t} ' )$ of functions on $\R$ which has $K^T$ as its reproducing kernel. We call this space a local bandlimit space or a sampling space. As part of Theorem \ref{tvbspace}, we prove:
\begin{thm*}
Let $\mu : [0,1) \rightarrow [0,1)$ be any smooth parametrization of $[0,1)$. The local bandlimit space $\K (T)$ is a reproducing kernel Hilbert subspace of $L^2 (\R , (\mu \circ \tau ) ' (t) dt)$, where $\tau = t^{-1}$, the compositional inverse of the spectral function of $T$. Any $f \in \K (T)$ obeys the sampling formulas:
$$ f(t) = \sum f (t_n (\theta) ) K ^T (t , t_n (\theta)  ); \quad \quad t_n (\theta) := t(n +\theta), \ \theta \in [0,1). $$
\end{thm*}
As before, $t$ denotes the spectral function of $T$, and $t_n (0) = t(n) = t_n$. This theorem shows that the spaces $\K (T)$ all have the same special reconstruction properties as the Paley-Wiener spaces of bandlimited functions: there is a one-parameter family of sampling sequences $\mbf{t} _\theta := ( t_n (\theta ) )$ (which cover the real line exactly once, see Lemma \ref{Rcover}), so that any $f \in \K (T)$ can be reconstructed perfectly from its samples taken on $\mbf{t} _\theta$.  Theorem \ref{tvbspace} further shows that $T$ is unitarily equivalent to an operator $M^T \in \mc{S} ^R ( \K (T) ) $ which acts as multiplication by the independent variable on its domain. This is one natural functional model for elements of $\mc{S} ^R$, and it provides a natural definition of a time-varying low-pass filter as the orthogonal projection of a raw signal onto a local bandlimit space (Definition \ref{tvlowpass}).  We show that the classical Paley-Wiener spaces are a special case of local bandlimit spaces in Example \ref{PW}.

The overall goal of the remaining sections is to connect the theory of the local bandlimit spaces $\K (T)$ to the classical theory of Hardy spaces of analytic functions in the upper half-plane, as well as to the spectral theory of self-adjoint multiplication operators on $L^2$ spaces associated to purely atomic measures on $\R$.  These theories will provide valuable new tools and insights that will lead to a natural definition of time-varying bandlimit (Definition \ref{TVbanddefn}), and will enable us to calculate any sampling sequence $\mbf{t} _\theta$ from the knowledge of an initial bandlimit pair of sequences $(\mbf{t} , \mbf{t} ')$. It is necessary to know these sequences in order to apply the above sampling formulas for the local bandlimit space $\K (T)$.

Section \ref{Hardy} develops the representation theory of $T \in \mc{S} ^R$ as multiplication by $z$ in a meromorphic model subspace of the Hardy space $H^2 (\C ^+ )$ of the upper half-plane \cite{Martin-dB,AMR}. We also review the Liv\v{s}ic characteristic function $\Theta _T$ of any $T \in \mc{S} ^R$, a complete unitary invariant for $\mc{S} ^R$ (Subsection \ref{Lcharfun}). This is a meromorphic inner function, \emph{i.e.} a bounded analytic function on $\C ^+$ which has a meromorphic extension to the whole complex plane and has unit modulus on the real axis. Corollary \ref{charfunspec} will show that the sampling sequences $\mbf{t} _\theta := ( t_n (\theta ) )$ are the solutions to $\Theta _T (t_n (\theta ) ) = e^{i2\pi \theta}$, and this will provide one method of calculating the entire family of sampling sequences $\mbf{t} _\theta; \ \theta \in [0,1)$ given an initial bandlimit pair of sequences $(\mbf{t} , \mbf{t} ' )$. Theorem \ref{TVmultiplier} will connect the theory of local bandlimit spaces to the theory of meromorphic model spaces of Hardy space (as well as to deBranges spaces of entire functions) by showing that any local bandlimit space, $\K (T)$, is the image of the meromorphic model space, $K(\Theta _T) := H^2 (\C ^+ ) \ominus \Theta _T H^2 (\C ^+ )$, under multiplication by a fixed function $M(t)$, and that this multiplication defines an onto isometry. It will follow, in particular, that elements of our local bandlimit spaces are, up to multiplication by a fixed function, meromorphic (in fact entire). Subsection \ref{TVbandsection} applies the theory of automorphisms of the unit disk to motivate and construct our definition of time-varying bandlimit, Definition \ref{TVbanddefn}.

Section \ref{Measmodel} develops a third class of models (or representations) for $\mc{S} ^R$ as multiplication by the independent variable in a measure space $L^2 (\R , \Ga )$, where $\Ga $ is a purely atomic positive measure whose atoms have no finite accumulation point (this is Theorem \ref{Hergmod}). We apply this model to compute formulas for the Liv\v{s}ic characteristic function of any $T \in \mc{S} ^R$ in terms of any of the sampling sequences $\mbf{t} _\theta = (t_n (\theta ) ), \ \theta \in [0,1 )$  associated to the self-adjoint extensions $T_\theta$ of $T$, see Corollary \ref{ACinner}. While the Aleksandrov-Clark measure representation of any contractive analytic function on $\C ^+$ is well-known, our identification of the weights of the purely atomic Aleksandrov-Clark measures for the meromorphic inner Liv\v{s}ic function $\Theta _T$ of $T \in \mc{S} ^R$ with the derivatives of the spectral function of $T$ may be novel. In Corollary \ref{corodie} we show that the spectral function $t$ is the unique solution to a first order ordinary differential equation obeying a certain initial condition. Given a local bandlimit space $\K (\mbf{t} , \mbf{t} ' ) = \K (T) $, both Corollary \ref{ACinner} and Corollary \ref{corodie} provide formulas for computing the sampling sequences $\mbf{t} _\theta = (t_n (\theta ) )$ from the knowledge of the initial bandlimit pair $(\mbf{t}, \mbf{t} ' )$.

\subsection{Symmetric linear transformations} \label{Notation}

Let $\H$ be a separable Hilbert space. Let $T$ be a (typically unbounded) linear transformation $T$ with with domain $\dom{T} \subset \H$.

\begin{defn}
    The linear transformation $T$ is called:
    \bn
        \item \emph{symmetric} if
        $$ \ip{Tx}{y} = \ip{x}{Ty} \quad \quad \forall x,y \in \dom{T}.$$
        \item \emph{densely defined} if $\dom{T}$ is dense in $\H$.
        \item \emph{simple} if there is no non-trivial proper subspace $S \subset \H$ so that the restriction of $T$ to $\dom{T} \cap S$ is self-adjoint.
        \item \emph{regular} if $T -t$ is bounded below on $\dom{T}$ for all $t \in \R$.
        \item \emph{closed} if the graph of $T$ is closed in $\H \oplus \H$.
    \en
    The \emph{deficiency indices}, $(n_+ , n_- )$ of $T$ are defined as $$ n _\pm := \dim{\ker{T^* \mp i}}.$$
    We will use the notation $\mc{S} $ to denote the family of all closed simple symmetric linear transformations with equal indices $(1,1)$ defined on a domain
in some separable Hilbert space. $\mc{S} ^R $ will denote the subfamily of all closed regular simple symmetric transformations with indices $(1,1)$ and similarly we
define $\mc{S} ( \H ), \mc{S} ^R (\H )$. Note that any symmetric $T$ always has a minimal closed extension, so there is no loss of generality in assuming that $T$ is closed \cite{Glazman}. We will call $T$ a symmetric operator if and only if $T$ is densely defined.
\end{defn}

Consider the map $$ b(z) := \frac{z-i}{z+i}, $$ with compositional inverse
$$ b^{-1} (z) = i \frac{1-z}{1+z}.$$ The map $b$ is an analytic bijection of the open upper half-plane $\C ^+$ onto the open unit disk $\D$. Moreover
$b$ is a bijection of the real line $\R$ onto $\T \sm \{1 \}$, the unit circle minus a point.

Let $\mc{V}  $ denote the family of all completely non-unitary (c.n.u.)
partial isometries with deficiency indices $(1,1)$ acting on a separable Hilbert space. Here the defect or deficiency indices of a partial isometry $V$ are defined by $n_+ := \dim{ \ker{V}}$ and $n_- := \dim{\ran{V} ^\perp}$. As shown in several standard texts \cite{Glazman,RS}, the map $T \mapsto b(T)$ defines a bijection of $\mc{S} _n$ (closed simple symmetric linear transformations with indices $(n,n)$) onto $\mc{V} _n$. Namely, given any $T \in \mc{S} _n$ one can define $b(T)$ as an isometric linear transformation from $\ran{T+i}$ onto $\ran{T-i}$. We can then view $V= b(T)$
as a partial isometry on $\H$ with initial space $\ker{V} ^\perp = \ran{T+i}$. Conversely given any $V \in \mc{V} _n$, one can define $b^{-1} (V)=T$ on the domain
$\ran{(V-I) V^* V }$, and then $T \in \mc{S} _n $ and $T = b^{-1} (b (T))$.

\subsection{Self-adjoint extensions} \label{saextsect}

Given $T \in \mc{S} $ let $V = b(T) \in \mc{V}$. One can construct a $\mc{U} (1)$ parameter family of
unitary extensions of $V$ as follows. Fix vectors $\phi _\pm$ of equal norm such that $$ \phi _+ \in \ker{V} = \ker{T^* -i} = \ran{T+i} ^\perp, $$ and $$\phi _- \in \ran{V} ^\perp = \ker{T^* +i} = \ran{T-i} ^\perp.$$ Define
 \be U (\alpha) := V + \frac{\alpha}{\| \phi _+ \| ^2} \ip{\cdot}{\phi_+} \phi _-; \ \alpha \in \T \quad \quad \mbox{and} \quad U_\theta := U(e^{i2\pi\theta}); \ \theta \in [0,1), \label{uniextensions} \ee where $\T$ is the unit circle in the complex plane. The set of all $U (\alpha)$ (or $U_\theta$) is the one-parameter family of all unitary extensions of $V$ on $\H$. The $U (\alpha)$ extend $V$ in the sense that $U (\alpha ) V^* V = V$ for all $\alpha \in \T$, they agree with $V$ on its initial space. We write $V \subseteq U(\alpha )$ to denote that $U(\alpha )$ extends $V$ in this way. Similarly, the subset notation $T \subset S$ for closed linear transformations $T,S$ denotes that $\dom{T} \subset \dom{S}$ and $S | _{\dom{T}} =T$, \emph{i.e.} $S$ is an extension of $T$.

We then define
  $$ T (\alpha ) := b^{-1} (U (\alpha )), \quad \quad T_\theta = T (e^{i2\pi \theta } ), $$ so that $T \subset T(\alpha ) \subset T^*$ for all $\alpha \in \T$. The functional calculus implies that each $T(\alpha) $ is a densely defined self-adjoint operator if and only if $1$ is not an eigenvalue of $U (\alpha )$, and the set of all $T(\alpha )$ (for which this expression is defined) is the set of all self-adjoint extensions of $T$. Note the assumption that $V$ be c.n.u. implies that $1$ is an eigenvalue to at most one $U(\alpha )$.

\subsection{Reproducing kernel Hilbert Spaces} \label{RKHStheory}
We will use basic reproducing kernel Hilbert space theory throughout this paper \cite{Paulsen-rkhs}.

Recall that a reproducing kernel Hilbert space (RKHS), $\K$, on some
set $X \subset \C$ is a Hilbert space of functions on $X$ with the property that point evaluation at any $x \in X$ defines a bounded linear functional, $\delta _x$, on $\K$. By the Riesz representation lemma, for any $x \in X$ there is a unique \emph{point evaluation vector} $K_x \in \K$ so that for any $F \in \K$,
$$ F(x) = \delta _x (F) = \ip{K_x}{F} _\K.$$ (All inner products are assumed to be conjugate linear in the first argument.) The  \emph{reproducing kernel} of $\K$ is the function
$K : X \times X \rightarrow \C$ defined by:
$$ K(x,y) := \ip{K_x}{K_y} _\K, $$ and one usually writes $\H (K) := \K$.
This reproducing kernel, $K$, is a \emph{positive kernel function} on $X \times X$, \emph{i.e.}, it has the property that for any finite set $\{ x_k \} _{k=1} ^N \subset X$,
the $N \times N $ matrix:
$$  [ K (x_k , x_j ) ] _{ 1 \leq k,j \leq N} \geq 0, $$ is non-negative. The classical theory of RKHS of Aronszajn and Moore (see \emph{e.g.} \cite{Paulsen-rkhs}) shows that there is a bijective correspondence between positive kernel functions $K$ on $X \times X$ and RKHS $\H (K)$ on $X$. That is, given any positive kernel function $K$, one can construct a RKHS $\H (K)$ which has $K$ as its reproducing kernel.

\subsection{Multipliers between RKHS} \label{multiplier}

Let $\H (k), \H (K)$ be two RKHS of $\C$-valued functions on some set $X$ with reproducing kernel functions $k, K$, respectively.  A function $F : X \rightarrow \C$ is called a \emph{multiplier} from $\H(k)$ into $\H (K)$
if $F h \in \H (K)$ for any $h \in \H (k)$. That is, $F$  is a multiplier if and only if multiplication by $F$,
$$ (M_F h) (x ) = F(x) h(x); \quad \quad x \in X, $$ defines a linear \emph{multiplication map}, $M_F :\H (k) \rightarrow \H (K)$. Let $\mr{Mult} (\H (k) , \H (K) )$ denote the set of all multipliers from $\H (k)$ into $\H (K)$. Standard functional analytic arguments show that (identifying $F$ with $M_F$)
$ \mr{Mult} (\H (k) , \H (K) ) \subseteq \L (\H (k) , \H (K) )$, and that $\mr{Mult} (\H (k) , \H (K) )$ is closed in the weak operator topology. The following elementary facts about multipliers will be useful (see, \emph{e.g.} \cite{GMR,AMR,Paulsen-rkhs}):
\begin{lemma} \label{multcri}
   A bounded linear map $M : \H (k) \rightarrow \H (K)$ is a multiplication map if and only if there is a function $m: X \rightarrow \C$ so that $$ M ^* K _x = \ov{m(x)} k_x; \quad \quad x \in X. $$ $M$ and $m$ satisfy this equation if and only if $m$ is a multiplier and $M = M_m$ is the corresponding multiplication map. A function $m : X \rightarrow \C$ belongs to $\mr{Mult} (\H (k) , \H (K) )$ if and only if there is a $C>0$ so that
$$ M(x) k(x,y) \ov{M (y)} \leq C K(x,y), $$ as positive kernel functions on $X \times X$. The function $m$ is an onto isometric multiplier if and only if equality holds with $C=1$.
\end{lemma}

Consider the case where $X \subseteq \C $, and assume that $\H (k) , \H (K)$ are such that $k_z , K_z \neq 0$ for any $z \in X$. Further suppose that there are linear transformations
$Z_k \in \mc{S}  (\H (k) ), \  Z_K \in \mc{S} (\H (K))$ which act as multiplication by the independent variable $z$. As in the case of bounded multipliers, it is easy to check that one always has
\be \ker{ Z_k ^* - \ov{z} } = \bigvee k _z; \quad \quad z \in X. \label{pevaleigen} \ee

\begin{lemma} \label{funmodelmult}
    Let $\H (k), \H (K)$ be RKHS on $X \subseteq \C$ with symmetric multiplication operators $Z_k ,Z_K \in \mc{S} (\H (k) ), \mc{S} (\H (K) )$ as above. A bounded linear map, $M : \H(k) \rightarrow \H (K)$, is a multiplication map if and only if $M Z_k \subset Z_K M$, \emph{i.e.}
if and only if $M \dom{Z_k} \subset \dom{Z_K}$ and $ M Z_k h = Z_K M h$ for all $h\in \dom{Z_k}$. An onto isometry $M : \H (k) \rightarrow \H (K)$ is a multiplier if and only if $M Z_k M^* = Z_K$.
\end{lemma}

\begin{lemma} \label{multext}
An onto isometry $M : \H (k) \rightarrow \H (K)$ is a multiplier if and only if given any self-adjoint extension
$Z ' $ of $Z_k$, $M Z ' M^*$ is a self-adjoint extension of $Z_K$.
\end{lemma}

\begin{defn} \label{funmodel}
Throughout this paper, a \emph{model} for $T \in \mc{S}$ is a pair
$(\hat{T} , \H )$, with $\H$ a separable (or finite dimensional) Hilbert space and $\hat{T} \in \mc{S} (\H )$, so that
$T$ is unitariliy equivalent to $\hat{T}$.

A model $(\hat{T} , \H )$ will be called a \emph{functional model} for $T \in \mc{S}$ if $\H = \H (k)$
is a reproducing kernel Hilbert space of functions on $X \subseteq \C$ with $\R \subseteq X$ and $\hat{T} = Z_k \in \mc{S} (\H (k) )$ acts as multiplication by the independent variable $z$ on its domain.
\end{defn}

The above lemmas imply:

\begin{cor} \label{multexist}
If $(Z_k , \H (k) )$, and $(Z_K , \H (K) )$ are two functional models for  $T \in \mc{S}$, then they are equivalent: there is a unitary multiplier $M: \H (k) \rightarrow \H (K)$ intertwining $Z_k$ and $Z_K$. Conversely, if $(Z_k , \H (k))$ is a functional model for $T \in \mc{S}$ and $M : \H (k) \rightarrow \H (K)$ is a unitary multiplier, then $(Z _K := M Z_k M^* , \H (K) )$ is a functional model for $T$.
\end{cor}

\begin{cor} \label{Kramer}
    Let $(Z_k , \H(k))$ be a functional model for $T \in \mc{S}$ where $\H(k)$ is a RKHS on $X \supseteq \R$.
If one of the self-adjoint extensions, $Z_k ' $ of $Z_k$ has spectrum equal to the closure of a discrete set of eigenvalues of multiplicity one, $\sigma (Z_k ' ) = \ov{ \{ t_n \} }$, then $\{ k_{t_n} \}$ is a total orthogonal set in $\H (k)$, and any $h \in \H (k)$ obeys the sampling formula:
$$ h(z ) := \sum _{n} f(t_n) \frac{ k (z ,t_n)}{k(t_n , t_n )}. $$
\end{cor}

This paper focuses on the representation theory of the class $\mc{S} ^R$ of regular, symmetric linear transformations with deficiency indices $(1,1)$. As we will see in the upcoming section, any symmetric $T \in \mc{S} ^R$ admits a natural functional model, $(M^T , \K (T) =\H (K^T))$ and the spectra of every self-adjoint extension, $T _\theta $ of $T$ is a discrete sequence of simple eigenvalues with no finite accumulation point.  The above corollary then implies that $\K (T)$ obeys a one-parameter family of Shannon-type sampling formulas (see Theorem \ref{tvbspace}).

\section{Abstract Functional Analysis model} \label{FAmodel}

In this section we present an abstract functional analytic approach to spaces of functions obeying a time-varying bandlimit as developed in \cite{Kempf2004,Hao2011}. The local bandlimit spaces, $\K (T); \ T \in \mc{S} ^R$, constructed in this section will be central to our approach and definition of time-varying bandwidth.

To construct a space of functions obeying a time-varying bandlimit, the input data is two sequences $\mbf{t} = (t_n )$ and
$\mbf{t} ' = (t_n ')$ obeying the following properties:

\begin{defn} \label{TVpairdef}
Let $\mathbb{F}$ be a subset of the integers $\Z$ of the
form $\mathbb{F} = \N \cup \{ 0 \} , - \N \cup \{ 0 \} , \Z$ or $\{ 0,1 , ... , N \}$.
Let $\mbf{t} =(t_n) \subset \R$ and $\mbf{t} ' = (t_n ') \subset [ 0 , \infty )$ be two sequences indexed by $\mathbb{F} \subset \Z$ with the following
properties:
\bn
    \item $ \mbf{t}$ is a strictly increasing sequence, $  t_n < t_{n+1},$ with no finite accumulation point. \\
    \item The two sequences $\mbf{t}, \mbf{t} '$ are compatible in the sense that:
    \be \sum \frac{t_n '}{1+t_n ^2 } < \infty. \label{TVpair} \ee
\en
Any such pair $(\mbf{t}, \mbf{t} ' )$ is called \emph{admissible}.
We say that an admissible pair $(\mbf{t} , \mbf{t} ' )$ is a \emph{time-varying bandlimit pair} or more simply, a \emph{bandlimit pair} if $\mbf{t} '  \subset (0, \infty )$, \emph{i.e.} $t_n '> 0 $ for all $n$. A bandlimit pair is called \emph{finite} if $\sum t_n ' < \infty$. Otherwise, if a bandlimit pair obeys $ \sum t_n ' = + \infty$ it is called \emph{infinite}.
\end{defn}

As proven in \cite{Kempf2004,Hao2011} (see also \cite[Theorem 2]{Martin-sym}) one has

\begin{lemma} \label{Rcover}
    Given any $T \in \mc{S} ^R$, fix a pair of equal norm deficiency vectors $\phi _\pm$ (and hence a parametrization, $T_\theta; \ \theta \in [0,1)$ of the self-adjoint extensions of $T$) the spectrum $\sigma (T _\theta )$ of each
self-adjoint extension $T _\theta$, $\theta \in [0,1)$ is
$$ \sigma (T _\theta ) := ( t_n (\theta ) ) =: \mbf{t} _\theta, \quad \mbox{where} \quad t_n (\theta ) = \sigma (T _\theta ) \cap [t_n , t_{n+1} ).$$ For each $\theta \in [0 , 1 )$, $\mbf{t} _\theta$ is a strictly increasing sequence
of eigenvalues of multiplicity one with no finite accumulation point, $\sigma (T _\theta ) \cap \sigma (T _\beta) = \emptyset$ for $\theta \neq \beta$ and
$$ \bigcup _{\theta \in [0 , 1 ) } \sigma (T _\theta ) = \R. $$ That is, the spectra of all self-adjoint extensions cover the real line exactly once.
\end{lemma}

\begin{remark}
    In the above references this theorem is proven assuming $T$ is densely defined. Working with the unitary extensions $U(\alpha) $ of $V = b(T)$, it is not
difficult to prove the obvious analogue of this result holds in general. Namely if $V$ is any c.n.u. (completely non-unitary) partial isometry with defect indices $(1,1)$, its inverse Cayley transform $T:= b^{-1} (V)$ is always a simple symmetric linear transformation with domain $\dom{T} = (I-V) \ker{V} ^\perp$, and it is not difficult to show that $T$ is densely defined if and only if $1$ is not an eigenvalue of any unitary extension $U (\alpha )$ of $V$, see \emph{e.g.} \cite[Theorem 3.0.9, Appendix I]{Glazman} or \cite[Theorem 3.1.2]{Martin-dB}. If $T \in \mc{S} ^R$ is not densely defined then the spectra of each unitary extension $U (\alpha)$ of
$V = b(T)$ can be arranged as a strictly increasing sequence of simple eigenvalues $( \alpha _n (\theta ) ) \subset \T$ (increasing in angle), where $\T$ denotes the unit circle. There is no overlap between the spectra of different extensions, and the spectra of all extensions cover the unit circle exactly once. It follows that $1$ is an eigenvalue of exactly one unitary extension $U$ of $V$, and one cannot take the inverse Cayley transform of this particular $U$ to obtain a densely defined self-adjoint extension of $T$. In the case where $T$ is densely defined $1$ is not an eigenvalue of any unitary extension of $V$ and the spectra of all unitary extensions cover $\T \sm \{ 1 \}$ exactly once.

This technical issue of when $T$ is or is not densely defined does not complicate the analysis or affect proofs in any significant way \cite{Martin-ext,Habock2001} and we will typically work under the assumption that $T$ is densely defined.
\end{remark}

\begin{remark}
The fact that $T_\theta$ has exactly one eigenvalue between any two eigenvalues of $T_0 = T(1)$ can be proven using Kre\v{\i}n's alternating eigenvalue theorem or using the theory of meromorphic model subspaces of Hardy space (this fact will also be a direct consequence of Theorem \ref{phasefun}) \cite{GG,Martin-dB}. The resulting ordered real \emph{sampling sequence} $\mbf{t} _\theta := ( t_n (\theta ) ); \ \theta \in [0,1)$ is then strictly increasing with no finite accumulation point.
\end{remark}

\begin{defn} \label{specfundef}
    Given $T \in \mc{S} ^R $ and a fixed choice of equal norm deficiency vectors $\phi _\pm \in \ker{T^* \mp i }$, the \emph{spectral function} of $T$,
$t : [a, b) \rightarrow \R$ is the strictly increasing bijection defined by
    \be t (n+ \theta) := t_n (\theta), \ee where
    \be a :=  \inf \{ n | \ t_n \in \sigma (T_0 ) \}, \quad \mbox{and} \quad   b := 1 + \sup \{ n | \ t_n \in \sigma (T_0 ) \}.\ee
\end{defn}

\begin{thm}{ (\cite[Corollary 5.1.7]{Martin-dB}, \cite{Martin-sym})}
    $t : [a,b) \rightarrow \R$ is a strictly increasing smooth bijection (onto $\R$) such that $t ' (s) > 0$. Moreover $t$ has a locally analytic extension about any point $s \in (a,b)$.
\end{thm}

Ultimately this follows from the fact that
$$ U (z) := V + \frac{z}{\| \phi _+ \| ^2} \ip{\phi_ +}{\cdot} \phi _-, $$ where $V = b(T)$ defines an entire operator-valued function.  In \cite{Martin-sym}, elementary Banach algebra techniques were used to establish the existence of such a function $t$, however a much simpler proof follows as a consequence of the representation theory of $\mc{S} ^R $ as multiplication by the independent variable in a meromorphic model subspace of Hardy space \cite[Corollary 5.1.7]{Martin-dB}. This will also follow from Theorem \ref{phasefun} of Section \ref{Hardy}.

\begin{lemma}
    Let $\phi _\pm , \psi _\pm$ be any two equal-norm pairs of deficiency vectors for $T \in \mc{S} ^R$. If $t_\phi , t_\psi$ are the corresponding spectral functions of $t$ then there is a $\zeta \in (-1 , 1)$ so that
$t_\phi (s) = t_\psi (s  -\zeta )$.
\end{lemma}

This is easily verified using the formula (\ref{uniextensions}). This defines an equivalence relation on spectral functions and the equivalence classes are a complete unitary invariant for $\mc{S} ^R$ \cite{Martin-sym}.  We will discuss a more useful unitary invariant for $\mc{S}$, the Liv\v{s}ic characteristic function, in Subsection \ref{Lcharfun}.

The following was first developed in \cite{Kempf2004}, see also \cite{Hao2011}:

\begin{thm} \label{FApic}
Let $(\mbf{t}, \mbf{t} ')$ be an admissible pair, and choose a self-adjoint operator $T'$ on a separable Hilbert space $\H$ with simple spectrum $\mbf{t}$ (each element of $\mbf{t}$ is a simple eigenvalue).  Fix a vector $\phi _+ \in \H$ so that if $ \{ \psi _n \}$ is any orthonormal eigenbasis of $T'$, $T' \psi _n = t_n \psi _n$, the coefficients of $\phi _+$ in this basis satisfy:
$$ | \ip{\psi _n}{\phi _+ } | ^2 = \frac{t_n ' }{1 + t_n ^2}. $$  Let $\phi _-$ and $\{ \phi _n \}$ be the unique vector and choice of orthonormal eigenbasis of $T'$ so that \be \phi _{\pm} = \sum \frac{ \sqrt{t_n ' } }{ t_n \mp i } \phi _n \in \H. \label{defectdef} \ee
Then there is a unique symmetric linear transformation $T$, $\dom{T} \subset \H$, with defect indices $(1,1)$ and no essential spectrum so that $\phi _\pm \in \ker{T ^* \mp i}$ are equal norm deficiency vectors for $T$, and with respect to the corresponding parametrization of self-adjoint extensions, $T ' = T_0$. \\
If $t_n ' >0$ for all $n$, \emph{i.e.} if $(\mbf{t} , \mbf{t} ' )$ is a bandlimit pair, then $T \in \mc{S} ^R (\H )$ is simple and regular. $T$ is densely defined if and only if
$(\mbf{t} , \mbf{t} '  )$ is an infinite bandlimit pair, \emph{i.e.} if and only if $\sum t_n ' = + \infty $. \\

Conversely suppose that $T \in \mc{S} ^R$, fix a choice of equal norm deficiency vectors $\phi _\pm \in \ker{T^* \mp i}$ (and hence a spectral function for $T$) and define $\mbf{t} = (t_n (0) ) = ( t ( n) ) $ and
$\mbf{t} '  := (\frac{\| \phi _+ \| ^2}{\pi} t ' (n) ) $. Then the pair of sequences $(\mbf{t} , \mbf{t} ')$ is a bandlimit pair. This bandlimit pair is finite if and only if $T$ is not densely defined.
\end{thm}

\begin{remark}
\bn
    \item In \cite{Kempf2004}, it was assumed that $(\mbf{t}, \mbf{t} ' )$ is an infinite bandlimit pair. Theorem \ref{FApic} above contains additional new information on how spectral properties of the symmetric linear transformation $T \in \mc{S}$ depend on properties of the admissible pair $(\mbf{t} , \mbf{t} ' )$.
    \item It was further assumed in \cite{Kempf2004} that the symmetric $T$ in the second half of Theorem \ref{FApic} is densely defined. Working in the setting of partial isometries and the unit circle, it is easy to extend the above result to the general case.
\en
\end{remark}

\begin{defn}
Let $(\mbf{t} , \mbf{t} ' )$ be a bandlimit pair, let $\phi _{\pm} \in \ker{T^* \mp i}$ be fixed as in equation (\ref{defectdef}) above, and let $t$ be the spectral function for $T$ fixed by the choice $\phi _\pm$.  For any $\theta \in [0,1 )$ define the pair of real sequences $(\mbf{t} _\theta , \mbf{t} _\theta ' )$ by $\mbf{t} _\theta = (t_n (\theta ) )$, $\mbf{t} _\theta ' = (t_n ' (\theta ))$,  where
$$ t_n (\theta ) := t ( n + \theta ), \quad \quad \mbox{and} \quad \quad t_n ' (\theta) := \frac{\| \phi _+ \| ^2 }{\pi} t' (n+\theta ). $$
By the next corollary each pair $(\mbf{t} _\theta, \mbf{t} _\theta ')$ is a bandlimit pair, and we will call $(\mbf{t} _\theta , \mbf{t} _\theta '); \ \theta \in [0,1 )$ a family of bandlimit pairs.
\end{defn}

With the above notation, $t_n = t_n (0)$, $t_n ' = t_n ' (0)$.

\begin{cor}
    Let $(\mbf{t} , \mbf{t} ' )$, and $\phi _\pm$ be fixed as above. The pair of sequences $(\mbf{t} _\theta , \mbf{t} _\theta ' )$ is a bandlimit pair for every $\theta \in [0,1)$. The symmetric operator $T^{(\theta )} \in \mc{S} ^R$ constructed as in Theorem \ref{FApic} from the data $(\mbf{t} _\theta , \mbf{t} _\theta ' )$ and $\phi _\pm$ is independent of $\theta$, $T^{(\theta)} = T^{(0)} = T$.
\end{cor}
 This is not difficult to verify. Starting with the data $(\mbf{t} , \mbf{t} ' )$ and $\phi _\pm$, the first part of Theorem \ref{FApic} guarantees the existence of a unique $T \in \mc{S} ^R$. Replacing $\phi _\pm$ by $\psi _\pm := e^{-i2\pi \theta } \phi _\pm$ and applying the second part of Theorem \ref{FApic} will show that $(\mbf{t} _\theta , \mbf{t} _\theta ' )$ is also a bandlimit pair.  This new choice of equal norm defect vectors amounts to re-parametrizing the self-adjoint extensions of $T$ by a constant shift of the parameter.

We provide the constructive half of the proof below to establish the new statements relating spectral properties of the constructed $T \in \mc{S}$ to properties of the admissible pair $(\mbf{t}, \mbf{t} ' )$:

Let $\mbf{t} = (t_n )$ be any strictly increasing sequence of real numbers with no finite accumulation point. Let $T(1) = T_0$ be a densely defined self-adjoint operator
on $\H$ so that $\sigma (T) = \{ t_n \}$ and each $t_n$ is an eigenvalue of multiplicity one. Let $\{ \psi _n \}$ be an orthonormal basis of eigenvectors corresponding to
the eigenvalues $t_n$. Now choose any sequence $\mbf{t} ' = (t_n ') $ with the same index set so that the pair $(\mbf{t} , \mbf{t} ')$ is admissible, namely
$$ \sum \frac{t_n '}{1+ t_n ^2 } < \infty. $$
This assumption ensures that
$$ \phi _\pm := \sum \frac{\sqrt{t_n '} }{t_n \mp i } \psi _n, $$ defines a pair of vectors of finite and equal norm in $\H$. Let
$$ \dom{T} := \{ \phi \in \dom{T(1)} | \ \ip{\phi _+}{(T(1) +i ) \phi } = 0 \}. $$ In other words, given $\phi = \sum c_n \psi _n \in \dom{T(1)}$ we have
that $\phi \in \dom{T}$ if and only if $$ \sum _n \sqrt{t_n '} c_n = 0. $$
Let $T := T(1) | _{\dom{T}}$. Clearly $T$ is a symmetric linear transformation in $\H$ with domain $\dom{T}$. By construction we have that
$\phi _\pm \perp \ran{T \pm i }$ so that $T$ is symmetric with indices $(n_+, n_-)$ and $n_\pm \geq 1$.

\begin{lemma} \label{bb}
    The symmetric linear transformation $T$ is such that $T-t$ is bounded below on $\ker{T -t} ^\perp \cap \dom{T}$ for any $t \in \R$ and $T$ has deficiency indices $(1,1)$.
\end{lemma}

\begin{proof}
First note that if $\phi \in \ker{T-t}$ then $\phi \in \dom{T} \subset \dom{T(1)}$ so that $\phi = \psi _n$ and $t =t_n$ for some $n$.
If $(T-t)$ is not bounded below on $\ker{T-t} ^\perp \cap \dom{T}$ then it follows that $(T(1) -t)$ is not bounded below on $\ker{T(1)-t} ^\perp \cap \dom{T(1)}$.
It follows that $t$ is in the essential spectrum of $T(1)$ which contradicts our assumption that the spectrum of $T(1)$ is a sequence of eigenvalues
of multiplicity one with no finite accumulation point.

Suppose that the deficiency indices of $T$ are not $(1,1)$. Since the eigenvalues of $T$ are contained in the eigenvalues of $T(1)$ it follows that there is a
maximal subset $(\psi _{n_k} ) \subset (\psi _n)$ so that the $\psi _{n_k} \in \dom{T}$. Removing $$ S := \bigvee \psi _{n_k}, $$ from $\H$ shows that we can write
$$ T = T' \oplus \hat{T} (1) \quad \quad \mbox{on} \quad (\H \ominus S) \oplus S, $$ where $\hat{T} (1) \subset T(1)$ is self-adjoint and $T'$ is simple symmetric. Recall here that $\bigvee$ denotes closed linear span. It further follows
that $T'-t$ is bounded below on $\dom{T'} = \dom{T} \cap (\H \ominus S)$ so that $T '$ is regular. Standard results show that $\dim{\ker{(T') ^* -z}}$ is constant
for $z \in \C$ \cite[Section 78]{Glazman}, \cite{RS}. It follows that the deficiency indices of $T$ are equal to the deficiency indices of $T'$ and these are $n = n_\pm$. Now suppose that $n>1$ and consider $t_j \in \sigma (T(0))$. Then $\ker{T^* -t_j}$ is $n-$dimensional where $n>1$. It follows that there is a non-zero vector $\psi  \in \ran{T -t_j } ^\perp $ such that $\psi  \perp \psi _j$.
Hence, $$ \psi = \sum _{k \neq j} c_k \psi _k, $$ so that $0 = (T^* -t_j ) \psi = \sum _{k \neq j} c_k (t_k -t_j) \psi _k.$ It follows that
$$ \sum _{k \neq j} |c_k| ^2 |t_k -t_j | ^2 =0,$$
which is satisfied if and only if $\psi =0$, a contradiction.
\end{proof}

\begin{lemma} \label{dd}
    $T$ is densely defined if and only if $\sum t_n ' = + \infty$.
\end{lemma}

\begin{proof}
    First if $\sum t_n ' < \infty $ then $f := \sum \sqrt{t_n '} \psi _n \in \H$ and clearly $f \perp \dom{T}$ so that $\dom{T}$ is not dense.

Conversely suppose that $\dom{T}$ is not dense and $ 0 \neq g := \sum g_n \psi _n \perp \dom{T}$. Then define
$$ g_+ := (T(1) -i) ^{-1} g =  \sum \frac{g_n}{t_n -i} \psi _n, $$ and observe that for any $\psi \in \dom{T}$,
$$ \ip{g_+}{(T+i) \phi} =0, $$ so that $g_+ \in \ran{T+i} ^\perp$. By the last lemma $\ran{T+i} ^\perp = \bigvee \phi _+$ is one dimensional
so that $g_+ = c \phi _+$. It follows that
$$ g = (T(1) -i ) g_+ = c (T(1) -i ) \phi _+ = c \sum \sqrt{t_n '} \psi _n \in \H $$ so that
$$ \sum t_n ' <\infty. $$
\end{proof}

\begin{lemma} \label{reg}
    The symmetric linear transformation $T$ is simple (and regular) if and only if $t_n ' > 0 $ for all $n$, \emph{i.e.} if and only if $(\mbf{t} , \mbf{t} ')$ is
a bandlimit pair of sequences.
\end{lemma}

\begin{proof}
    If $t_k ' =0$ then it is easy to see that $$ \ip{\phi _+}{(t_k +i) \psi _k } =0, $$ which implies that $\psi _k \in \dom{T}$ and $T$ is not simple.

Conversely if $T$ is not simple then there is a point $t \in \R$ such that $t$ is either an eigenvalue of $T$ or in the approximate point
spectrum. If $t$ is an eigenvalue of $T$ then it is an eigenvalue of $T(1) \supset T$ so that $t = t_k$ for some $k$. Since the spectrum of $T(1)$ consists of eigenvalues of multiplicity one, it would follow that $\psi _k \in \dom{T}$ so that as above,
$$ \sqrt{t'_k} = \ip{\phi _+}{(t_k +i) \psi _k} = \ip{\phi _+}{(T+i)\psi _k } = 0.$$

Now suppose that $t$ belongs to the approximate point spectrum of $T$. Then $b(t) \in \T \sm \{ 1 \}$ belongs to the approximate point spectrum and hence the essential
spectrum of the partial isometry $V=b(T)$. The essential spectrum is invariant under compact perturbations and $T$ has deficiency indices $(1,1)$,
 so that $U(1) = b(T(1))$ is a rank-one unitary perturbation of $V$. It follows that $b(t)$ is in the essential spectrum of $b(T(1))$ and so $t$ is in the essential spectrum of $T(1)$.
This contradicts our assumption that the spectrum of $T(1)$ is a sequence of eigenvalues of multiplicity one with no finite accumulation point.

Finally if $T$ is simple then Lemma \ref{bb} implies that $T-t$ is bounded below for all $t \in \R$ so that $T$ is also regular.
\end{proof}

This concludes half of the proof of Theorem \ref{FApic}. Namely we have shown that any bandlimit pair of sequences $(\mbf{t} , \mbf{t} ' )$ can be used to construct a linear transformation $T \in \mc{S} ^R$.  We refer to \cite{Kempf2004} for the converse proof that any $T \in \mc{S} ^R$ defines a bandlimit pair of sequences.

\subsection{Local bandlimit spaces} \label{LBS}

Our goal now is to apply Theorem \ref{FApic} to construct an abstract functional model for any $T \in \mc{S} ^R$. In particular, we will construct a reproducing kernel Hilbert space, $\K (T) ; \ T \in \mc{S} ^R$, which embeds isometrically into $L^2 (\R , d\la)$ for a family of positive measures $\la$ which are equivalent to Lebesgue measure. This space, called a \emph{local bandlimit space}, will be a function space with the same special sampling and reconstruction properties as the Paley-Wiener spaces of bandlimited functions. The Paley-Wiener space of $A$-bandlimited functions will be an example of one such space (Example \ref{PW}).

Let $(\mbf{t} , \mbf{t} ' )$ be a bandlimit pair. As in the previous subsection, let $T_0 = T(1)$ be a self-adjoint operator on $\H$ with orthonormal basis $\{ \psi _n \}$ and spectrum $\sigma (T_0) := \mbf{t}$.
Choose $$ \phi _+ := \sum \frac{\sqrt{t_n '} }{t_n -i } \psi _n \in \H, $$ and construct $T \in \mc{S} ^R (\H )$ with deficiency vectors $\phi _+ \in \ker{T^* -i}$, and $\phi _- := b(T_0) \phi _+ \in \ker{T^* +i}$ as before. Recall that this choice of deficiency vectors fixes a family of self-adjoint extensions $T_\theta$, $\theta \in [0,1)$ of $T$ (see Subsection \ref{saextsect}), as well as the choice of spectral function $t$ of $T$ defined by $ t (n + \theta ) := t_n (\theta)$ where $\sigma ( T _\theta ) = (t_n (\theta ) )$. Recall that we define $a:= \inf \{ n | \ t_n \in \sigma ( T_0 ) \}$ and $b:= 1 + \sup \{ n | \ t_n \in \sigma (T_0 )\}$ so that $[a, b)$ is the domain of the spectral function $t$.

\begin{defn} \label{phasedef}
The \emph{phase function} $\tau : \R \rightarrow [a,b )$ of $T$ is the compositional inverse of the spectral function $t$ (fixed by a choice of equal-norm deficiency vectors).
\end{defn}

It follows that $\tau : \R \rightarrow [a,b)$ is injective,
strictly increasing and obeys $\tau ' > 0$ (since $t$ has these properties). Theorem \ref{phasefun} will imply that $\tau$ has an analytic extension to a neighbourhood of $\R$.

\begin{defn}
Given $T \in \mc{S} _R$ and a fixed deficiency vector pair,$\phi _\pm$, let $\{ \phi _n (\theta ) | \ \theta \in [0,1) \}$ be any fixed family of orthonormal eigenbases of the family $T_\theta$ of self-adjoint extensions of $T$ (the parameter $\theta$ is fixed by the choice of $\phi _\pm$), $T _\theta \phi _n (\theta ) = t_n (\theta ) \phi _n (\theta )$.

For any $t \in \R$ let $\lfloor t \rfloor $ denote the integer part of $t$,  let $[t] := t - \lfloor t \rfloor \in [0,1)$, and define
\be \phi _t := \phi _{\lfloor \tau (t) \rfloor } ( [t ] ) ; \quad \quad [t] = t - \lfloor t \rfloor, \ee where the phase function $\tau$ (and the spectral function) is fixed by the choice $\phi _\pm$.
\end{defn}

Observe that, for any $s \in \R$,
 \ba T^* \phi _s & = & T _{[s]} \phi _{\lfloor \tau (s) \rfloor } ( [\tau (s) ] ) = t _{\lfloor \tau (s) \rfloor } ( [\tau (s) ] ) \phi _{\lfloor \tau (s) \rfloor } ( [\tau (s) ] ) \nn \\
 & = & t (\lfloor \tau (s) \rfloor + [\tau (s) ] ) \phi _s = t (\tau (s) ) \phi _s \nn \\
 & = & s \phi _s. \nn \ea

\begin{prop} \label{kernelprop}
For any $T \in \mc{S} ^R$, and fixed equal-norm deficiency vectors $\phi _{\pm} \in \ker{T^* \mp i}$,
there is a choice of orthonormal eigenbases $\{ \phi _n (\theta ) | \ \theta \in [ 0 ,1 ) \}$ of eigenvectors for $T_\theta$ so that if $\phi _t := \phi _{\lfloor \tau (t) \rfloor} ([\tau (t) ])$, then
\ba  K ^T (t,s) & := &\ip{\phi _t}{ \phi _s} ;  \nn \\
& = & f(t) (-1) ^{\lfloor \tau (t) \rfloor} \left( \sum _k \frac{t' (k)}{(t-t_k)(s-t_k)} \right) (-1) ^{\lfloor \tau (s) \rfloor} f(s); \quad \quad s,t \in \R, \label{kernelformula}\ea is a smooth, real-valued, positive kernel function on $\R \times \R$ where
\be f(t) := \left( \sum _n \frac{t' (n)}{(t-t_n) ^2 } \right) ^{- \frac{1}{2}}. \ee
\end{prop}

\begin{lemma}{ (\cite[Theorem 8, Chapter 3]{Hao2011})} \label{defcoeff}
Let $(\mbf{t}, \mbf{t} ' )$ be a bandlimit pair. Let $T_0$ be the corresponding self-adjoint operator acting in $\H$ with $T_0 \supset T \in \mc{S} ^R (\H)$ constructed as in Theorem \ref{FApic}. For each $\theta \in [0,1)$, let
$\{ \psi _n (\theta ) \}$ be an arbitrary orthonormal basis of eigenvectors to $T_\theta$ with eigenvalues $t_n (\theta ) = t ( n + \theta )$.
Expand the deficiency vector $\phi _+$ in the basis $\{ \psi _n (\theta ) \}$ as
$$ \phi _+ = \sum \frac{c_n (\theta )}{t_n (\theta ) -i} \psi _n (\theta ). $$ \\
Then the coefficients $c_n (\theta )$ obey
$$ | c_n (\theta ) | ^2 = \frac{\| \phi _+ \| ^2}{\pi} \left. \frac{d t_n (\beta )}{d\beta} \right| _{\beta = \theta}.$$
\end{lemma}
    The proof of this lemma in \cite{Hao2011} assumes that the bases $\{ \psi _n (\theta) \}$
can be chosen so that the coefficients $c_n (\theta )$ are continuous functions of $\theta \in [0,1)$.
Although this fact is not immediately obvious in the current setup, it will follow easily from Hardy space theory, see Remark \ref{contcoeff}.

\begin{remark} \label{normpair}
If the bandlimit pair $(\mbf{t} , \mbf{t}')$ is normalized so that
$$ \sum \frac{t_n'}{1+t_n ^2} = \pi, $$ (we can always rescale the sequence $\mbf{t} '$ so that this is the case and this does not change the symmetric operator $T$) then $\| \phi _+ \| ^2 = \pi$ and we obtain $$ t_n ' (\theta ) = t ' ( n +\theta ); \quad \theta \in [0,1). $$ In particular,
it follows that if $(\mbf{t} , \mbf{t} ' )$ is a normalized bandlimit pair, then every bandlimit pair $(\mbf{t} _\theta , \mbf{t} _\theta ' )$ is normalized.
\end{remark}

\begin{proof}{ (of Proposition \ref{kernelprop})}
For simplicity we will simply write $K(t,s) :=  K ^T (t,s) = \ip{\phi _t }{\phi _s}$. Given
any choice of orthonormal bases $\{ \phi _n (\theta) | \ \theta \in [ 0 ,1) \}$, it is clear that $K(t,s)$
will be a positive kernel function on $\R \times \R$. It remains to show that these bases can be chosen so that $K$ is smooth and given by the formula (\ref{kernelformula}).

Given the fixed bandlimit pair $(\mbf{t} , \mbf{t} ' )$, we can, as in the proof of Theorem \ref{FApic}, choose a self-adjoint operator $T_0$ with discrete spectrum $\mbf{t} = (t_n )$
consisting of eigenvalues of multipicity one with normalized eigenvectors $\{ \psi _n \} $.

Also as before we can define
$$ \phi _{\pm} := \sum \frac{\sqrt{t_n '}}{t_n \mp i } \psi _n, $$ and construct a symmetric $T \in \mc{S} ^R$ so that $\phi _\pm \in \ker{T^* \mp i }$ are equal-norm deficiency vectors for $T$. For each $\theta \in [0,1 )$,
let $\{ \phi _n (\theta ) \}$ be an orthonormal basis of eigenvectors for the self-adjoint extension $T_\theta$ with corresponding eigenvalues $t_n (\theta )$.  It will be convenient to choose $\phi _n (0) =: \phi _n = (-1) ^n \psi _n$, so that we can expand $\phi _+$ in each basis as
$$ \phi _+ := \sum c_n (\theta ) \frac{ e^{i2\pi \theta }}{t_n (\theta ) -i } \phi _n (\theta ), $$
for some coefficients $c_n (\theta )$, where $c_n (0) = (-1) ^n \sqrt{t_n ' }$.

For any $\theta \in [0 , 1)$, the unitary extension $U_\theta$ of $V=b(T)$ is
$$ U_\theta = U (e ^{i2\pi \theta} ) = V + \frac{e^{i2\pi \theta}}{\| \phi _+ \| ^2 } \ip{ \phi _+ }{\cdot } \phi _-. $$ Since $U_\theta = b (T_\theta)$,
$$ U _\theta \phi _n (\theta ) = \frac{ t_n (\theta ) - i}{t_n (\theta ) +i } \phi _n (\theta ), $$ and it follows that
$$ \phi _- = e^{-i2\pi \theta} U_\theta \phi _+ = \sum _n  \frac{c_n (\theta )}{t_n (\theta ) + i } \phi _n (\theta ); \quad \quad \theta \in [0,1).$$

In order to compute $\ip{ \phi _t }{\phi _s }$ for any $s,t \in \R$, we need to evaluate
$\ip{ \phi _n (\theta ) }{ \phi _m (\beta )}$ for any $\theta, \beta \in [0,1)$ and any $n, m$ in the index set $\F \subseteq \Z$. First consider
\ba \ip{\phi _n (\theta )}{ (U_\beta - U_\theta ) \phi _m (\beta )} & = & \left( \frac{t_m (\beta ) -i}{t_m (\beta ) +i } - \frac{t_n (\theta ) - i}{t_n (\theta ) +i} \right) \ip{\phi _n (\theta)}{\phi _m (\beta ) } \nn \\
& =& \frac{2i ( t_m (\beta ) - t_n (\theta ) )}{(t_m (\beta ) +i)(t_n (\theta ) + i )} \ip{\phi _n (\theta )}{\phi _m (\beta )}. \nn \ea
Using that
$$ U_\beta - U_\theta = \frac{( e^{i2\pi \beta} - e^{i2\pi \theta} )}{\| \phi _+ \| ^2 }  \ip{\phi _+  }{\cdot} \phi _-, $$ this same expression can be evaluated differently:
$$ \ip{\phi _n (\theta )}{ (U_\beta - U_\theta ) \phi _m (\beta )}  =
\frac{( e^{i2\pi \beta} - e^{i2\pi \theta} )}{\| \phi _+ \| ^2 }
\frac{e^{-i2\pi \beta} \ov{c_m (\beta )} c_n (\theta ) }{(t_m (\beta ) +i )(t_n (\theta) +i )}. $$

Equating these two expressions yields
\ba \ip{\phi _n (\theta )}{\phi _m (\beta )} & = & \frac{1-e^{i2\pi (\theta - \beta)} }{2i(t_m (\beta ) - t_n (\theta))} \frac{\ov{c_m (\beta)} c_n (\theta )}{\| \phi _+ \| ^2 }. \nn \\
& = & \frac{ e^{i\pi (\theta -\beta ) }}{ \| \phi _+ \| ^2 } \frac{ \sin \left( \pi (\beta - \theta) \right)}{t_m (\beta) - t_n (\theta ) } \ov{c_m (\beta )} c_n (\theta ). \nn \ea

By the previous lemma, there are phases $w_n (\theta ) \in \R$ so that
$$ c_n (\theta ) = \frac{ \| \phi _+ \| }{\sqrt{\pi}} \sqrt{t ' (n +\theta ) } e^{iw_n (\theta )}.$$
For $\theta \in (0,1)$ we are free to choose the numbers $w_n (\theta )$ arbitrarily, since the normalized basis vectors $\phi _n (\theta )$ can be re-defined to absorb the unimodular constants $e^{iw_n (\theta )} \in \T$. For any $\theta \in [0 ,1)$ we choose
$ w_n (\theta ) := -\pi (\theta +n ),$ so that
$$ e^{i w_n (\theta )} = e^{-i\pi \theta } (-1) ^n. $$ This fixes the orthonormal bases $\{ \phi _n (\theta ) \} $ uniquely (since the choice of $\phi _n (0)$ has already been fixed), and we obtain that
\be \ip{ \phi _n (\theta ) }{\phi _m (\beta )} = \frac{(-1) ^{n+m} \sin \left( \pi (\beta - \theta )\right)}{t_m (\beta ) - t_n (\theta )} \frac{ \sqrt{ t' (n +\theta ) t' (m + \beta ) }}{\pi}. \label{form1} \ee

Fix any $\alpha \in [ 0 ,1)$ so that $\theta \neq \alpha$. Then, expanding in the orthonormal basis $\{ \phi _n (\alpha ) \}$ yields
\ba 1 & = & \ip{ \phi _n (\theta ) }{\phi _n (\theta ) } \nn \\
& = & \sum _k \frac{ \sin ^2 \left( \pi (\alpha -\theta ) \right)}{\left( t_k (\alpha ) - t_n (\theta ) \right)^2 } \frac{ t' (n +\theta ) t' (k +\alpha ) }{\pi ^2}. \nn \ea Solving for $t' (n +\theta )$ yields the
formula
\be t' (n +\theta ) = \frac{\pi^2}{\sin ^2 \left( \pi (\alpha - \theta ) \right) } f_\alpha (t_n (\theta ) ) ^2 , \label{form2} \ee
where
$$ f_\alpha (t) := \left( \sum _k \frac{ t' (k+\alpha )}{ (t - t_k (\alpha )) ^2 } \right) ^{-\frac{1}{2}}.$$
Expanding $ \ip{ \phi _n (\theta ) }{\phi _m (\beta )}$ in the $\{ \phi _k (\alpha ) \}$ basis and subsituting in the formulas (\ref{form1}) and (\ref{form2}) then yields the formula
\ba & & \ip{ \phi _n (\theta ) }{\phi _m (\beta )}   =  \sum _k \ip{ \phi _n (\theta ) }{\phi _k (\alpha ) } \ip{ \phi _k (\alpha ) }{\phi _m (\beta )} \nn \\
& = & \sum _k \frac{(-1) ^{n+m}}{\pi ^2 } \frac{t' (k+\alpha )}{ \left( t_n (\theta ) - t_k (\alpha )\right) \left( t_m (\beta ) - t_k (\alpha ) \right)} \sin \left( \pi (\alpha - \theta )\right) \sin \left( \pi (\alpha -\beta ) \right) \sqrt{ t' (n +\theta ) t' (m +\beta )} \nn \\
& = & \sum _k (-1) ^{n+m} \frac{ t' (k+\alpha )}{  \left( t (n + \theta ) - t (k +\alpha )\right) \left( t ( m+ \beta ) - t (k+ \alpha ) \right)} f_\alpha (t ( n+ \theta ) ) f_\alpha (t (m +\beta )). \nn \ea
By definition $\phi _t = \phi _{\lfloor \tau (t) \rfloor} ([ \tau (t) ] )$, for any $t \in \R$, so that
\ba K (t, s) & = & \ip{ \phi _{\lfloor \tau (t) \rfloor} ( \tau (t) - \lfloor \tau (t) \rfloor ) }{\phi _{\lfloor \tau (s) \rfloor} ( \tau (s) - \lfloor \tau (s) \rfloor) ) } \nn \\
& = & (-1) ^{\lfloor \tau (t) \rfloor } f _\alpha (t) \left( \sum _k  \frac{ t' (k +\alpha ) }{\left( t - t_k (\alpha ) \right) \left( s - t_k (\alpha ) \right) } \right) f _\alpha (s ) (-1) ^{\lfloor \tau (s) \rfloor}, \label{form3} \ea for any $\alpha \in [ 0 ,1)$. In particular, choosing $\alpha =0$ gives the claimed formula from the proposition statement. The function
$f_\alpha$ and the formula (\ref{form3}) for $K(t,s)$ are clearly infinitely differentiable for $t,s \notin \{ t_k (\alpha ) \}$. Since $\alpha \in [0, 1)$ is arbitrary, and the $ (t_k (\alpha ))$ cover the real line exactly once, we conclude that $K(t,s)$ is smooth, \emph{i.e.} infinitely differentiable in both arguments.
\end{proof}

\begin{defn} \label{smoothpara}
Let $\mu : [0 ,1 ] \rightarrow [0,1]$ be a strictly increasing bijection which is infinitely differentiable on $(0,1)$.
We will say that $\mu$ is a \emph{smooth parametrization} of $[0,1]$ provided that the extended bijection $\mu _e : \R \rightarrow \R$
defined by $\mu _e (t) := \lfloor t \rfloor + \mu ([t] )$ is smooth (infinitely differentiable) on $\R$ and $\mu ' (t) > 0$ is strictly positive.  We will simply write $\mu = \mu _e$ for this extension.
\end{defn}
\begin{defn} \label{TVkerneldefs}
Define the rescaled positive kernel function
$$ K^{(T; \mu )} (t,s) := \sqrt{(\mu \circ \tau )' (t)} K ^T (t,s) \sqrt{(\mu \circ \tau ) ' (s)}; \quad \quad t,s \in \R. $$
Let $(\mbf{t} , \mbf{t} ' )$ be a bandlimit pair of real sequences, and $T \in \mc{S} ^R$ the symmetric regular linear transformation corresponding to $(\mbf{t} , \mbf{t} ' )$ (and a choice of equal-norm deficiency vectors) by Theorem \ref{FApic}.

The \emph{sampling space} or \emph{local bandlimit space} of time-varying bandlimited functions, $\K (T)$, is the reproducing kernel Hilbert space (RKHS) of functions on $\R$ with reproducing kernel $K ^T$, $ \K (T) := \H (K ^T )$, where $K _T$ is as given in Proposition \ref{kernelprop}.
We will sometimes use the alternate notation $\K (T) = \K (\mbf{t} , \mbf{t} ')$. Similarly, given any smooth parametrization $\mu$ on $[0,1]$, the \emph{$L^2$ sampling subspace} or \emph{$L^2 $ local bandlimit subspace}, $\K  (T ; \mu) = \K _\mu (\mbf{t} , \mbf{t} ' ; \mu )$ is the RKHS of functions on $\R$ with reproducing kernel $K^{(T; \mu )}$.
\end{defn}

\begin{thm} \label{thm:mult}
    Let $(\mbf{t}, \mbf{t} ' )$ be a bandlimit pair of sequences and let $T \in \mc{S} ^R (\H)$ be the corresponding symmetric linear transformation (constructed as in Theorem \ref{FApic}).
The map $U ^T : \H \rightarrow \K (T)$ defined by
$$ U ^T \phi _t := K ^T _t ; \quad \quad t \in \R, $$  is an onto isometry obeying
$$ h^T (t) := (U^T h) (t) = \ip{\phi _t}{h}; \quad \quad h \in \H. $$
The image $U^T T (U^T)^* =: M ^T \in \mc{S} ^R (\K (T) )$ acts as multiplication by the independent variable $t$ on its domain $\dom{M^T} = U^T \dom{T}$.
\end{thm}

This theorem shows that the pair $(M^T , \K (T) )$ is a functional model for $T \in \mc{S} ^R$ in the sense of Definition \ref{funmodel}. Recall that the point evaluation vectors $K ^T _t$ for $t \in \R$ are as defined in Subsection \ref{RKHStheory}. Namely, recall that $K ^T _t \in \K (T)$ is the unique vector which obeys $$ \ip{K _t }{F} = F(t); \quad \quad \forall F \in \K  (T).$$

\begin{proof}
Since the set $\{ \phi _t \} _{t \in \R }$ contains orthonormal bases, $U^T$ is densely defined.
By the construction of $K^T$ in Proposition \ref{kernelprop},
\ba \ip{U^T \phi _t}{ U^T \phi _s} _{\K (T) } & = & \ip{K ^T _t}{K ^T _s} _{\K(T)} \nn \\
& =& K ^T (t,s) \nn \\
& =& \ip{\phi _t}{\phi _s} _\H. \nn \ea $U^T$ is onto since the point evaluation vectors $K^T _t$, for $t\in \R$ are dense in the RKHS $\K (T)$. Observe that for any $h\in \H$,
$$ (U^T h) (t) = \ip{K ^T _t}{U^T h} _{\K (T)} = \ip{\phi _t}{h} _\H. $$

The final assertion is similarly easy to check: For any $h \in \dom{T}$,
\ba (M ^T U^T h ) (t) & = & (U^T T h) (t) \nn \\
& =& \ip{K _t ^T }{ U^T T h} _{\K (T)} \nn \\
& = & \ip{ \phi _t }{ Th} _{\H} \nn  \\
& = & t \ip{\phi _t }{h} _\H = t (U^T h) (t). \nn \ea
\end{proof}

\begin{thm} \label{tvbspace}
    Let $\mu$ be any smooth parametrization of $[0,1)$. The sampling space $\K (T)$ is a reproducing kernel Hilbert subspace of $L^2 (\R , (\mu \circ \tau ) ' (t) dt )$ and the $L^2$ sampling subspace $\K (T;\mu)$ is a reproducing kernel subspace of $L^2 (\R )$. The map $U ^{(T; \mu )} : \H \rightarrow \K (T; \mu ) \subset L^2 (\R )$ defined by $$ (U^{(T; \mu)} h) (t) := \sqrt{(\mu \circ\tau) ' (t)} (U ^T h ) (t) =
\sqrt{\mu ' (\tau (t) ) \tau ' (t)} \ip{\phi _t}{h} _\H, $$ is an onto isometry.
For any $\theta \in [0 ,1 )$, $\{ K ^{(T;\mu)} _{t (n +\theta )} \} \subset \K (T; \mu)$ is an orthonormal basis of point evaluation vectors and these are
eigenvectors of  $M^{(T;\mu)} _\theta := U^{(T;\mu)} T_\theta (U^{(T;\mu)}) ^*$ to the eigenvalues $t_n (\theta) = t(n +\theta )$. This yields the sampling formulas:
$$ f(t) = \sum f(t_n (\theta ) ) K ^{(T;\mu)} (t , t_n (\theta ) ); \quad \quad f \in \K (T; \mu). $$ The symmetric linear transformation $M^{(T;\mu)} := U^{(T;\mu )} T (U^{(T;\mu)} ) ^*$ acts as multiplication by the independent variable on its domain in $\K (T ;\mu)$.
\end{thm}

\begin{remark}
Similar sampling formulas hold, of course, for the non-scaled sampling spaces $\K (T)$.
\end{remark}

\begin{proof}
    We know that $\{ \phi _n (\theta)  = \phi _{t ( n+\theta )} \}$ is an orthonormal basis of eigenvectors to $T_\theta$ with eigenvalues $t_n (\theta ) = t ( n +\theta  )$,
$\theta \in [0 ,1 )$.
It follows that for any $f \in \H$ and $\theta \in [0,1)$,
\ba (U^T T f) & = & \sum \ip{\phi _{t_n (\theta )}}{f} U ^T \phi _{t_n (\theta)} \nn \\
& = & \sum f (t_n (\theta) ) K ^T _{t_n (\theta)}. \nn \ea
Let $f^T := U^T f$. Using that the $\{ \phi _n (\theta ) \}$ and hence the $\{ K ^T _{t_n (\theta) } \}$ are an orthonormal basis we have that
\ba  \| U ^T f \| ^2 & =& \int _0 ^1  \| U^T f \| ^2 \mu ' (\theta ) d\theta \nn \\
& =& \int _0 ^1 \sum | f^T (t (n +\theta) | ^2 \mu ' (\theta ) d\theta \nn \\
& = & \int _a ^b |f^T (t (s)) | ^2 \mu ' (s) ds, \nn \ea where we have extended $\mu$ periodically to $[a,b)$ as in Definition \ref{smoothpara}. Change variables by setting $s = \tau (t)$ to obtain
$$ \| f^T \| ^2 = \intfty | f^T (t) ) | ^2 \mu ' (\tau (t) ) \tau ' (t) dt. $$
The rest of the claim is straightforward.
\end{proof}

\subsection{A time-varying low-pass filter}
The above theorem shows that any $L^2 $ sampling subspace $K (T ; \mu) \subset L^2 (\R )$ has similar sampling and reconstruction properties to the Paley-Wiener spaces of bandlimited functions.  Namely, any $L^2$ sampling space has a one-parameter family of total orthogonal sets of point evaluation vectors $\{ K ^{(T; \mu ) } _{t_n (\theta )} | \ \theta \in [0,1) \}$ where the discrete sets of sample points $\{ t_n (\theta ) \}$ cover the real line exactly once and have no finite accumulation point. Moreover, these spaces have several useful properties that make them practical for signal processing applications. First, the reproducing kernel, and hence the point evaluation vectors $K^{(T; \mu)} _t$ are all real-valued, so that their Fourier transforms are centred in frequency space, and this is a natural property one would like locally bandlimited functions to have. Secondly, since $\K (T ; \mu)$ is a subspace of $L^2 (\R)$, the best approximation in $\K (T ; \mu)$ to any raw signal $f_{raw} \in L^2 (\R)$ is simply the image of $f_{raw}$ under orthogonal projection onto $\K (T ;\mu)$. In classical signal processing, a \emph{low-pass filter} is a device or process that removes all frequencies from a raw signal greater than a fixed cutoff value, $A>0$. That is, the low-pass filter implements the orthogonal projection of $f_{raw}$ onto the Paley-Wiener space $B(A)$. By Theorem \ref{tvbspace} the projector onto the $L^2$ sampling space $\K (T ;\mu ) \subseteq L^2 (\R )$ can be expressed as either an integral or as a countable summation:
$$ P _ {\K (T ; \mu) } = \intfty K ^{(T;\mu)} _t \ip{ K^{(T; \mu) } _t}{\cdot} dt = \sum \frac{K^{(T;\mu) } _{t_n (\theta)}}{K^{(T;\mu)} (t_n (\theta ) , t_n (\theta) )} \ip{K^{(T; \mu)} _{t_n (\theta )}}{\cdot} ; \quad \theta \in [0,1). $$

\begin{defn} \label{tvlowpass}
    Given any $L^2$-local bandlimit space $\K (T ; \mu )$, the \emph{$(T;\mu)$ time-varying (low-pass) filter} is the orthogonal projection of $L^2 (\R)$ onto $\K (T ; \mu )$.
\end{defn}

Finally, as described in the introduction, the local bandlimit space $\K (T) = \K (\mbf{t} , \mbf{t} ' ;\mu )$ is completely determined by the bandlimit pair of real sequences $(\mbf{t}, \mbf{t} ' )$ (and the choice of parametrization $\mu$), and these sequences can be tailored to match the local frequency behaviour of any given set of raw signals.

We conclude this section by showing that the classical Paley-Wiener spaces of $A-$bandlimited functions, are, in fact, a special case of sampling spaces.

\begin{eg} \label{PW}
Consider the bandlimit pair $(\mbf{t}, \mbf{t} ' )$ where
$$ \mbf{t} = \left( \frac{n\pi}{A} \right) _{n \in \Z}, \quad \quad \mbox{and}, \quad \quad \mbf{t} ' = \left( \tanh (A) \frac{\pi}{A} \right) _{n \in \Z}. $$

Applying the identity
$$ \sum _{n \in \Z } \frac{1}{n^2 + t^2} = \frac{\pi}{t} \coth (\pi t ), $$ it is easy to verify that this pair is normalized as in Remark \ref{normpair}.

Rewrite the expression
\ba  \sum _k \frac{ t' (k) }{ (t- t_k ) (s- t_k) } & = & \frac{\tanh A}{s-t} \sum _k \left( \frac{1}{ \frac{At}{\pi} -k } - \frac{1}{\frac{As}{\pi} -k } \right) \nn \\
& = & \pi \tanh (A) \frac{ \cot (At) - \cot (As) }{s-t}, \nn \ea where we have applied the series identity
$$ \sum _k \left( \frac{1}{t-k} - \frac{1}{s-k} \right) = \pi \cot (\pi t) - \pi \cot (\pi s).$$

Also recall the function $f(t)$ has the form
$$ f(t) ^{-2}  =  \sum _n \frac{t_n '}{(t-t_n) ^2 } = A \pi \tanh (A) \csc ^2 (At), $$ where we have again applied a standard trigonometric series formula. It follows that
\ba (-1) ^{\lfloor \tau (t) \rfloor} f(t) & = & (-1) ^{\lfloor \tau (t) \rfloor} \frac{1}{\sqrt{A\pi \tanh (A) } } | \sin (At ) | \nn \\
& = & (-1) ^{\lfloor \tau (t) \rfloor + \lfloor \frac{At}{\pi} \rfloor} \frac{1}{\sqrt{A\pi \tanh (A) } }  \sin (At ) \nn \\
& = & \frac{1}{\sqrt{A\pi \tanh (A) } } \sin (At ). \nn \ea The last line follows
since  $t_n = \frac{n\pi}{A} \leq t < t_{n+1} = \frac{(n+1) \pi}{A}$ implies that
$$ n = \lfloor \tau (t) \rfloor = \lfloor \frac{At}{\pi} \rfloor. $$ If $K$ is the positive kernel function corresponding to the pair $(\mbf{t} , \mbf{t} ' )$, then,
\ba K (t ,s ) &  = & f(t) (-1) ^{\lfloor \tau (t) \rfloor} \left( \sum _k \frac{t' (k)}{(t-t_k)(s-t_k)} \right) (-1) ^{\lfloor \tau (s) \rfloor} f(s) \nn \\
& = &   \frac{\sin (At)}{\sqrt{A\pi \tanh (A)}} \left( \frac{\pi \tanh (A) ( \cot (At) - \cot (As) }{s-t} \right) \frac{\sin (As)}{ \sqrt{A\pi \tanh (A)}} \nn \\
& = & \frac{ \sin \left( A (t-s) \right)}{A (t-s)}. \nn \ea

Using Fourier theory, it is easy to check that the reproducing kernel for $B(A)$ is
$$ k^A (t,s) := \frac{A}{\pi} \frac{ \sin \left( A (t-s) \right)}{A (t-s)}, $$ a constant multiple of the kernel for $\K (\mbf{t} , \mbf{t} ')$.
Lemma \ref{multcri} then implies that multiplication by the positive constant $\sqrt{ A / \pi }$ is a unitary operator from $\K (\mbf{t} , \mbf{t} ' )$ onto $B(A)$.
\end{eg}

\section{Hardy space model} \label{Hardy}

The local bandlimit spaces $\K (T)$ provide a functional model for any $T \in \mc{S} ^R$.  An equivalent functional model can be constructed using the theory of meromorphic model spaces of the Hardy space of the upper half-plane. This classical theory will provide a valuable perspective on the local bandlimit spaces, and, in particular, will motivate a precise definition of time-varying bandwidth.

Let $H^2 := H^2 (\C ^+ )$ be the Hardy space of analytic functions in the upper half-plane. Recall that $H^2$ can be defined as the space of all analytic functions $h \in \C ^+$
so that the norm
$$ \| h \| ^2 := \sup _{y>0} \intfty |h(t+iy) | ^2 dt < \infty, $$ is finite. The classical theory of Hardy spaces shows that any $h \in H ^2$ has non-tangential boundary values on $\R$ which exist almost everywhere with respect to Lebesgue measure, and that the identification of $H^2$ functions with their non-tangential limits defines an isometric inclusion of $H^2$ in $L^2 (\R)$ \cite{Hoff}.  Equivalently,
$H^2 (\C ^+ ) = \H (k)$ is the unique RKHS corresponding to the sesqui-analytic Szeg\"{o} kernel:
$$ k(z,w) := \frac{i}{2\pi} \frac{1}{z-\ov{w} } ; \quad \quad z,w \in \C ^+. $$

Similarly one can define $H^\infty = H ^\infty (\C ^+)$ as the Banach space of all analytic functions which are bounded in $\C ^+$. As before non-tangential boundary values define an isometric embedding of $H^\infty$ into $L^\infty$.

Recall that any $\Theta \in [H^\infty ] _1$, the closed unit ball of $H^\infty$, is called \emph{inner} if $$ |\Theta (t ) | = 1; \quad \quad \mbox{a.e.} \ t \in \R, $$
\emph{i.e.}, if $\Theta$ has unimodular non-tangential boundary values almost everywhere with respect to Lebesgue measure on the real line. Let $S$ denote the operator of multiplication by $$ b(t) := \frac{t-i}{t+i}, $$ restricted to $H^2 (\C ^+ )$. It is easy to check that $S$ is an isometry on $H^2$, called the \emph{shift}. Under the canonical unitary transformation of $H^2 (\C ^+)$ onto the Hardy space of the disk $H^2 (\D)$, $S$, is conjugate to the operator of multiplication by $z$, the shift on $H^2 (\D)$. The shift operator plays a central role in the study of Hardy spaces \cite{Hoff,Nik2012,Ross2006CT,Sarason-dB}.

A classical theorem of Beurling-Lax shows that a subspace $M \subset H^2$ is invariant for $S$ if and only if
$$ M = \Theta H^2, $$ for some inner function $\Theta$ \cite{Hoff}.  The corresponding \emph{model space} $K(\Theta ) := H^2 \ominus \Theta H^2$ is then invariant for the backward shift, $S^*$, and cyclic for $S$ \cite{Ross2006CT,Nik2012}. Any model space $K(\Theta )$ is a RKHS of analytic functions on $\C ^+$ with reproducing kernel
 \be k^\theta (z,w) := \frac{i}{2\pi} \frac{1 - \Theta (z) \ov{\Theta (w)}}{z-\ov{w}}; \quad \quad z,w \in \C ^+. \label{modkernform} \ee  We will be primarily interested in the case where $\Theta$ is a \emph{meromorphic inner function}, \emph{i.e.} an inner function which has a meromorphic extension to the entire complex plane $\C$.

\subsection{An analytic functional model}

As shown in \cite{Martin-dB, AMR}, given any inner function $\Phi \in H^\infty$, one can define
$$ \dom{Z ^\Phi} := \{ f \in K(\Phi ) | \ zf(z) \in K (\Phi ) \}. $$ Let $M$ denote the self-adjoint operator of multiplication by $t$ in $L^2 (\R)$.
Then for any inner function $\Phi$, $$ Z ^\Phi := M | _{\dom{Z ^\Phi}} \in \mc{S} (K (\Phi )), $$ is a simple symmetric linear transformation in $K (\Phi )$ with deficiency indices $(1,1)$ \cite{Martin-dB,AMR}. The domain $\dom{Z ^\Phi }$ is not necessarily dense, in particular it is not dense if $K (\Phi )$ is finite dimensional, which occurs, for example, if $\Phi $ is a finite Blaschke product. For necessary and sufficient conditions for $Z^\Phi$ to be densely defined  see for example \cite{Livsic1946}, \cite[Corollary 3.1.3, Corollary 3.1.4, Theorem 5.0.9]{Martin-dB} or \cite[Appendix 1, Theorem 5.6]{Glazman}.

It is easy to check (as in Subsection \ref{RKHStheory}) that the $k _z ^\Phi$, are eigenvectors for $(Z ^\Phi) ^*$,
\be \ker{ (Z ^\Phi) ^* - \ov{z} } = \bigvee k_z ^\Phi ; \quad \quad z \in \C ^+. \label{defspaces} \ee

For the remainder of this section we assume that $\Phi$ is a meromorphic inner function, \emph{i.e.}, $\Phi$ has a meromorphic extension to $\C$. Since
$\Phi$ has unimodular non-tangential boundary values on $\R$ almost everywhere with respect to Lebesgue measure, it follows that $| \Phi (t) | = 1$ for all $t \in \R$ and that $\Phi$ is analytic in an open neighbourhood of $\R$. In particular, $K (\Phi )$ can be viewed as a RKHS on an open neighbourhood of $\ov{\C ^+ }$, so that the reproducing kernel formula (\ref{modkernform}) holds for all $z,w \in \ov{\C ^+} $ and the above formula (\ref{defspaces}) extends to all $t \in \R$.
As established in \cite{Martin-dB}, a symmetric linear transformation $S$ is regular and simple with deficiency indices $(1,1)$, \emph{i.e.} $S \in \mc{S} ^R $ if and only if it is unitarily equivalent to some $Z ^\Phi$ acting on $K (\Phi )$, where $\Phi$ is a meromorphic inner function (equivalently $(Z ^\Phi , K (\Phi ) )$ is a functional model for $T$).

\begin{thm}
    Any inner function $\Phi \in H^\infty (\C ^+ )$ which has a meromorphic extension to $\C$  has the form
\be \Phi (z) := \gamma e^{iaz} \prod  \frac{\ov{z _n}}{z_n} \frac{z-z_n}{z-\ov{z_n}} ; \quad \quad (z_n) \subset \C ^+, \label{meroinner} \ee where $a\geq 0$, $\ga \in \T$,  the $(z_n )$ have no finite accumulation point and obey the Blaschke condition
$$ \sum \frac{\im{z_n}}{|z_n| ^2} < \infty. $$
The symmetric linear transformation $Z ^\Phi \in \mc{S} ^R $ is densely defined if and only if either $a >0$ or $\sum \im{z_n} = + \infty$.
\end{thm}
This factorization formula follows easily from the Blaschke-singular factorization of inner functions \cite{Hoff}. The necessary and sufficient condition on meromorphic inner $\Phi$ so that $Z^\Phi \in \mc{S} ^R$ is densely defined is the special case of Liv\v{s}ic's criterion applied to $T \in \mc{S} ^R$, see \emph{e.g.} \cite[Theorem 5.0.9]{Martin-dB}. The Blaschke condition (the necessary and sufficient condition on the $\{ z_n \}$ so that the above product converges) combined with the assumption that the $z_n$ have no finite accumulation point is equivalent to the convergence of $\sum \left| \im{\frac{1}{z_n}} \right| = \sum \frac{\im{z_n}}{|z_n | ^2}.$

\begin{remark}
Meromorphic inner functions are related to deBranges functions (also called Hermite-Biehler functions). An entire function $E$ is called a deBranges function if $| E (z)  | > |E (\ov{z} ) | $ for all $z \in \C ^+$, and $\Phi$ is a meromorphic inner function if and only if $\Phi = E ^\dag / E$ for some deBranges function $E$ where $E^{\dag} (z) := \ov{E(\ov{z} )}$ \cite[Section 2.3]{Hav2003}, \cite[pg. 317-318]{Levin1964}. The theory of deBranges spaces of entire functions provides another equivalent functional model for $\mc{S} ^R$ \cite{dB,Martin-dB}.
\end{remark}

There a natural conjugation $C_\Phi$ on $K (\Phi)$ which commutes with $Z ^\Phi$ (see, for example
\cite{Martin-dB}, \cite{Sarason2007}, or \cite[Section 7.6]{GG}):

\begin{lemma}
    For any inner $\Phi$, define the anti-linear map $C_\Phi : K (\Phi ) \rightarrow K (\Phi )$ by
\be (C_\Phi k ^\Phi _w) (z) = \frac{1}{2\pi i} \frac{\Phi (z) - \Phi (w)}{z-w}. \label{conjker} \ee
This map extends to an anti-linear, idempotent surjective isometry (a conjugation) so that $Z ^\Phi C_\Theta \subset C_\Phi Z ^\Phi$:
$$ C _\Phi \dom{Z^\Phi } \subset \dom{Z  ^\Phi }, \quad \quad \mbox{and} \quad \quad C_\Phi Z ^\Phi \dom{Z^\Phi } = Z^\Phi C_\Phi \dom{Z ^\Phi}.$$
\end{lemma}

\begin{remark}
This lemma immediately implies that
$C_\Phi$ is an isometry from $\ker{ (Z ^\Phi ) ^* -\ov{z} }$ onto $\ker{(Z^\Phi ) ^* - z}$ so that
$$ \ker{ (Z^\Phi ) ^* - z} = \bigvee C_\Phi k_z ^\Phi; \quad \quad z \in \C ^+. $$

It follows that one can choose equal norm deficiency vectors $\phi _\pm \in \ker{ (Z^\Phi )^* \mp i }$ by
 \be \phi _+  = \wt{k} ^\Phi _i := - C_\Phi k_i ^\Phi, \quad \quad \mbox{and} \quad \quad  \phi _- =  -k ^\Phi _i = C_\Phi \phi _+. \label{defchoice} \ee
We will refer to this choice of deficiency vectors as the canonical choice.
\end{remark}

\begin{remark} \label{contcoeff}
    Any non-zero deficiency vector $\psi _+ \in \ker{(Z^\Phi) ^* -i}$ is a constant multiple of $\phi _+ = \wt{k} ^\Phi _i = - C_\Phi k_i ^\Phi$.
It follows easily from the formula (\ref{conjker}) for $\phi _+$ that if $$ a_n (\theta) := \frac{1}{\| k ^\Phi _{t_n (\theta )} \|} \ip{\psi _+}{k ^\Phi _{t_n (\theta)}}$$ are the coefficients
of $\psi _+$ in the total orthonormal basis of normalized point evaluation eigenvectors to $Z^\Phi _\theta$, that $a_n (\theta )$ is continuous as a
function of $\theta \in [0,1)$. This fact is used in the proof of Lemma \ref{defcoeff}.
\end{remark}

The above canonical choice (\ref{defchoice}) of deficiency vectors fixes the family of self-adjoint extensions $Z^\Phi (\alpha)$, $\alpha \in \T$. With this choice one can prove, \cite[Section 4.1]{Martin-dB}:

\begin{thm} \label{extspec}
Let $\Phi$ be a meromorphic inner function. Fix a family of
self-adjoint extensions $Z^\Phi (\alpha )$ of $Z^\Phi \in \mc{S} ^R$ by the above choice of deficiency vectors $\phi _+ = \wt{k} ^\Phi _i = -C_\Theta k_i ^\Phi$ and $\phi _- = - k_i ^\Phi$. Then the spectrum of $Z^\Theta (\alpha )$, $\alpha \in \T$ is
\be \sigma (Z^\Phi (\alpha) ) = \left\{ t \in \R \left| \ \Phi (t) = \frac{ \alpha + \Phi (i) }{1 + \alpha \ov{\Phi (i)}} \right. \right\}. \ee
\end{thm}

In particular, if $\Theta$ is a meromorphic inner function such that $\Theta (i) = 0$, then
$$ \sigma (Z^{\Theta} (\alpha ) ) = \{ t \in \R | \ \Theta (t) = \alpha \}. $$  Theorem \ref{extspec} is easily proven by applying the formula
$$ \dom{ Z^\Phi (\alpha )} = \dom{Z^\Phi} \bigvee \{  \phi _+ - \alpha \phi _- \}, $$ to determine when $k_t ^\Phi$ belons to the domain of $Z^\Phi (\alpha)$. (This formula follows as $Z^\Phi (\alpha )$ is the inverse Cayley transform of a unitary extension $U^\Phi (\alpha)$ of $b (Z^\Phi )$).

\begin{thm}{ (\cite[Theorem 5.13]{Martin-dB}, \cite[Problem 48]{dB})} \label{phasefun}
    If $\Phi \in H^\infty$ is inner and meromorphic then there is a strictly increasing function $\ga$ on $\R$ so that
$\Phi (t) = e^{i2\pi \ga (t)}$, $\ga ' (t) >0$ for all $t \in \R$ and $\ga$ has an analytic extension to a neighbourhood of $\R$.
\end{thm}

\begin{defn}
    Let $\Phi$ be a meromorphic inner function. The unique function $\ga$ such that $\Phi (t) = e^{i2\pi \ga (t) }$ as above and $\ga (0) \in [0,1)$ is called the \emph{phase function} for $\Phi$. The compositional inverse, $x$, of $\ga$, is called the \emph{spectral function} of $\Phi$.
\end{defn}

This next corollary follows readily from Theorem \ref{extspec}:
\begin{cor}
    If $\Theta$ is a meromorphic inner function obeying $\Theta (i) =0$, then the spectral and phase functions of $Z^\Theta \in \mc{S} ^R (K (\Theta ))$ fixed by the choice $\phi _+ = \wt{k} _i ^\Theta, \ \phi _- = -k_i ^\Theta$ are the spectral and phase functions for $\Theta$.
\end{cor}

\begin{remark} \label{specfunction}
Since $\tau$ is strictly increasing and smooth, $\ran{\tau} = (a,b) \subset \R$. We have defined $\tau$ so that $\tau (t_0 ) = \tau (t(0)) = 0$. It follows that if the strictly increasing spectral sequence $\mbf{t} = (t_n )$ of $Z ^\Theta _0$ is indexed by $\{ 0, 1, ... , N\}$, $\{ 0 \} \cup \N$, $\{-N, ... , 0 \}$, $-\N \cup \{ 0 \}$ or $\Z$ that the range of $\tau $ is $[0 , N+1 )$,$[ 0 , \infty)$, $(-N+1 , 0], (-\infty , 0]$ or $\R$ respectively (and then $\F = \Z \cap \ran{\tau}$). Since the set of all meromorphic inner $\Theta$ form a multiplicative semigroup (with unit $\Theta \equiv 1$), it follows that the set of all phase functions is an additive semigroup.
\end{remark}

\subsection{The Liv\v{s}ic characteristic function} \label{Lcharfun}

Given any $T \in \mc{S} $, one can define a contractive analytic function $\Theta _T$ on $\C ^+$, which is a complete unitary invariant for $T$, called the \emph{Liv\v{s}ic characteristic function}
of $T$  \cite{Livsic1946,Livsic1950}. We will prove in Corollary \ref{ACinner}, that given any family of bandlimit pairs $(\mbf{t} _\theta , \mbf{t} _\theta ')$, and corresponding symmetric operator $T \in \mc{S} ^R$ (fixed by a choice of equal-norm defect vectors $\phi _\pm$ as in Theorem \ref{FApic}), that the Liv\v{s}ic characteristic function $\Theta _T$ can be expressed solely in terms of any bandlimit pair $(\mbf{t} _\theta , \mbf{t} _\theta ' )$. This will yield new formulas for computing any pair $(\mbf{t} _\theta , \mbf{t} _\theta ' )$ from the knowledge of an initial pair $(\mbf{t} , \mbf{t} ' )$.

The Liv\v{s}ic characteristic function is defined as follows:
Let $\phi _{\pm } \in \ker{T^* \mp i}$ be fixed deficiency vectors of equal norm and let $\phi _z \in \ker{T^* -z}$, $z \in \C ^+$, be an arbitrary non-zero vector. Then
\be \label{Livdef} \Theta _T (z) := \frac{z-i}{z+i} \frac{\ip{\phi _z}{\phi _{+}}}{\ip{\phi _z}{\phi _-}}; \quad \quad z \in \C ^+.  \ee The characteristic function always vanishes at $z=i$, $\Theta _T (i) =0$. Two contractive analytic functions $\Theta _1, \Theta _2$ on $\C ^+$ are said to coincide if there is a unimodular constant $\alpha \in \T$ so that
$$ \Theta _1 = \alpha \Theta _2.$$ This defines an equivalence relation and the Liv\v{s}ic function is unique up to this notion of equivalence. A unique representative $\Theta _T$ in a given coincidence class is fixed by a unique choice of the deficiency vectors $\phi _\pm$. A different choice
of deficiency vectors $\psi _\pm = \alpha _\pm \phi _\pm$ where $\alpha _\pm \in \T$ yields a new Liv\v{s}ic function $\Theta ' _T = \alpha _- \ov{\alpha _+} \Theta _T$ which coincides with the first.

\begin{thm}{ (Liv\v{s}ic, \cite{Livsic1946})}
    Given any two $T_1, T_2 \in \mc{S} $, $T_1 \simeq T_2$ if and only if their characteristic functions coincide.

    The map from $T \mapsto \Theta _T$ is a bijection from unitary equivalence classes of $\mc{S}$ onto coincidence classes of contractive analytic functions on $\C ^+$ which vanish at $z=i$.
\end{thm}

It is straightforward to calculate that if $\Phi$ is inner then the Liv\v{s}ic characteristic function of $Z^\Phi \in \mc{S}$ (fixed by the choice of deficiency vectors $\phi _+ := \wt{k} _i ^\Phi$ and $\phi _- = -k _i ^\Phi$) is \cite{Martin-dB}:
$$ \Theta _{Z^\Phi} = \frac{\Phi - \Phi (i) }{1 - \Phi \ov{\Phi (i)}} =: F_{\Phi (i)} (\Phi ). $$ In particular, if $\Phi (i) = 0$ then $\Phi = \Theta$. Recall here that for any fixed $w \in \D$, the M\"{o}bius transformation,
\be F_w (z) := \frac{z-w}{1-z\ov{w}} \label{Mobius}, \ee is an analytic automorphism of the unit disk with compositional inverse $F _{-w}$. That is, $F_w$ is a bijection of $\ov{\D}$ onto itself which maps the circle $\T$ onto itself, and which is analytic in an open neighbourhood of $\ov{\D}$. The composition $F_w \circ \Phi$ for any inner $\Phi$ is called a \emph{Frostman shift} of $\Phi$ \cite{Frostman,Crofoot}. The above calculation and theorem of Liv\v{s}ic then imply \cite[Theorem 5.0.7]{Martin-dB}:

\begin{thm} \label{Hardymodelrep}
    A symmetric linear transformation $T \in \mc{S} $ belongs to $\mc{S} ^R $ if and only if $\Theta _T$ is a meromorphic inner function.
Given any $w \in \D$, $T$ is unitarily equivalent to $Z ^{F_w \circ \Theta _T }$ acting in the meromorphic model space $K (F_w \circ \Theta _T)$.
\end{thm}

\begin{remark}
    By the above theorem, given any $T \in \mc{S} ^R$ and $w \in \D$, $( Z ^{F_w \circ \Theta _T} , K (F_w \circ \Theta _T )$ is a functional model for $T$
(Definition \ref{funmodel}). By Lemma \ref{funmodelmult} for any $w \in \D$ there is an isometric multiplier from $K (F_w \circ \Theta _T )$ onto $K (\Theta _T )$.  This multiplier is referred to in the literature as a \emph{Crofoot transform} \cite{GR-model,Crofoot}. By comparing the kernel functions for $K(\Theta _T )$ and $K (F_w \circ \Theta _T )$, it is easy to check that this multiplier is given by the formula: $$ M_w (z) := \sqrt{ 1 - |w| ^2} (1 - \Phi (z) \ov{w} ) ^{-1}. $$

Lemma \ref{funmodelmult} of Subsection \ref{multiplier} also implies that for any $T \in \mc{S} ^R$, there is a unitary multiplier between the local bandlimit space $\K (T)$ and the meromorphic model space $K (\Theta _T )$. This multiplier will enable us to move freely between these two functional models for $T \in \mc{S} ^R$. The characteristic function $\Theta _T$ of $T$ will be calculated in the next section (see Corollary \ref{ACinner}), and we will compute this multiplier in Subsection \ref{Kmult}.
\end{remark}

\subsection{Analytic parametrizations}

Any analytic automorphism of the unit disk, $F_w$, for $w \in \D$ provides a smooth re-parametrization, $\mu _w$, of the unit interval $[0,1]$ in the sense of Definition \ref{smoothpara}. 

\begin{defn}
    For any $w \in \D$, the $w$-\emph{analytic parametrization} of $[0,1]$ is the function $\mu _w : [0,1] \rightarrow [0,1]$
defined by $$ e^{i 2\pi \mu _w (\theta )} := F_{w} (e^{i2\pi \theta} ) \ov{F_{w} (1)} = \frac{e^{i2\pi \theta} - w }{1 - \ov{w} e^{i2\pi \theta} } \frac{1 - \ov{w}}{1-w}, \quad \quad \mbox{and} \quad \mu _w (0) = 0. $$
Also define $\la _w : \R \rightarrow \R$ as the unique solution to: 
$$ e^{i2\pi \la _w (t)} := F_w (e^{i2\pi t } ); \ t \in \R \quad \mbox{and}  \quad \la _w (0) \in [0, 1). $$ 
\end{defn}

Several easy observations can be made: Since $F_{-w}$ maps the unit circle $\T$ onto itself and is analytic in an open neighbourhood of $\ov{\D}$, it follows that $\mu _{w}$ is a smooth bijection of $[0,1)$ onto itself and $\mu _w (0) =0$. It is then straightforward to calculate that
$$ \mu _w ' (\theta ) = \la _w ' (\theta ) =  \frac{1 -|w| ^2}{ | e^{i2\pi \theta} -w | ^2 } > 0, $$ so that $\mu _w$ is a smooth re-parametrization of $[0,1]$ in the sense of Definition \ref{smoothpara}, and $\la _w$ is a strictly increasing smooth bijection of $\R$ onto $\R$.
If $\theta _w := \la _w (0) \in [0,1)$ then (identifying $\mu _w$ with its periodic extension):
\be \mu _w (t ) = \la _w (t ) - \theta _w, \quad  \mbox{and} \quad  \mu _w ' ( t ) = \la _w ' ( t ); \ t \in \R. \label{argshift} \ee
It is also easy to see that if $n = \lfloor t \rfloor$, $[t] = t - \lfloor t \rfloor$, then
\be \la _w (t ) = \la _w ( [t ] ) + n. \label{periodic} \ee  Finally, since the compositional inverse of $F_w$ is $F_{-w}$, the compositional inverse of $\la _w$ is $\la _{-w}$.

\begin{cor} \label{charfunspec}
Given any meromorphic inner function $\Phi$, with $w:= \Phi (i)$, let $\Theta := F_w \circ \Phi$ so that $\Theta (i) =0$.
Fix the family $Z ^\Phi _\theta$, $\theta \in [0,1)$ of self-adjoint extensions of $Z^\Phi$ by the canonical choice of deficiency vectors (Equation \ref{defchoice}), and let $\ga , x$, and  $\tau , t$ be the phase and spectral functions for $\Phi$ and $\Theta$, respectively. Then, 
$$ \sigma (Z^\Phi _\theta )  =  \{ t ( \theta +n ) |   \ n \in \F = \Z \cap \ran{\tau} \},$$ and the corresponding set of eigenvectors
$$ \{ k_t ^\Phi | \ t \in \sigma (Z^\Phi _\theta )  \}, $$ is a total orthogonal set in $K (\Phi )$.
\end{cor}

The above corollary is an immediate consequence of Theorem \ref{extspec}, and the observation that $\Theta$ is chosen so that $\Theta (t) = e^{i2\pi \tau (t)}$,
and $\tau := \la _w \circ \ga$ so that $t = \tau ^{-1} = x \circ \la _{-w}$. As in Theorem \ref{tvbspace}, Corollary \ref{Kramer} implies that the meromorphic model spaces $K (\Phi )$ obey a one-parameter family of sampling formulas.

\begin{eg}{ (Paley-Wiener spaces)}

Recall that as shown in Example \ref{PW}, any Paley-Wiener space $B(A)$ is (up to rescaling) the local bandlimit space $\K ( \mbf{t} , \mbf{t} ' )$ with
$\mbf{t} = \left( \frac{n\pi}{A} \right)$ and $\mbf{t} ' = \left( \frac{\pi}{A} \tanh{A}  \right).$ Also by Theorem \ref{tvbspace}, if $\mu$ is any smooth parametrization of $[0,1]$, $\K (\mbf{t} , \mbf{t} ' )$ embeds isometrically as a subspace of $L^2 (\R , (\mu \circ \tau ) ' (t) dt )$.  Since the Paley-Wiener space $B(A)$ is a subspace of $L^2 (\R)$, it seems reasonable to expect that there should be a choice of smooth re-parametrization, $\mu$ of $[0,1]$ so that $(\mu \circ \tau) ' $ is a constant.

It is straightforward to check that $B(A) := e^{-iAz} K (e ^{i2Az} )$ is the image of a meromorphic model space under a unitary multiplier, and that this amounts to a shift in frequency space. It follows that there is a $Z \in \mc{S} ^R ( B(A) )$ which acts as multiplication by the independent variable, $z$, and that the Liv\v{s}ic characteristic function of $Z$ is (up to a unimodular constant)
$$ \Theta (z) = \frac{e^{i2Az} - e^{-2A}}{1- e^{-2A} e^{i2Az}} = F_{e^{-2A}} (e^{i2Az} ). $$ Equivalently $\Phi (z) := e^{i2Az} = F_{-e^{-2A}} (\Theta (z))$ with phase function $\ga (t) = \frac{A}{\pi} t$. 
If $\tau$ is the phase function of $\Theta$, it then follows that $\tau = \la _w \circ \ga $ with $w = e^{-2A}$, so that
$$ \frac{A}{\pi} = \ga ' (t) = \mu  _{-w} ' (\tau (t) ) \tau ' (t), $$ (recall $\mu _w ' = \la _w '$) and, upon rescaling by the constant $\sqrt{ A / \pi }$, $\K (\mbf{t} , \mbf{t} ' )$ embeds isometrically in $L^2 (\R )$, as expected.
\end{eg}

\subsection{Time-varying bandlimit} \label{TVbandsection}

The classical notion of bandlimit for any Paley-Wiener space $B(A)$ can be interpreted as a measure of the density of any of the Nyquist sampling lattices:
$x _n (\vartheta ) = (n + \vartheta ) \frac{\pi}{A}$, and $x _{n+1} (\vartheta ) - x_n (\vartheta ) = \frac{\pi}{A}$, $\vartheta \in [0,1)$. Recall, as in the previous example, $B(A) = e^{-iAz} K (e^{i2Az} )$. The phase function, $\ga$, of $\Phi (t) = e^{i2At} = e^{i2\pi \ga (t) }$ is simply $\ga (t) = \frac{A}{\pi} t$, and this is the compositional inverse of the function $x (n+\vartheta) := x_n (\vartheta ); \ \vartheta \in [0,1)$ (the spectral function of $\Phi$). It follows that the bandlimit is
$$ A = \pi \ga ' (t), $$ where $\ga$ is the phase function of the meromorphic inner function $\Phi (z) = e^{i2Az}$.

Working in analogy with the classical Paley-Wiener spaces of $A$-bandlimited functions, we can construct a precise and meaningful definition of time-varying bandwidth for any local bandlimit space $\K (\mbf{t} , \mbf{t} ' )$.  Let $\Phi$ be any meromorphic inner function obeying $\Phi (i) = w \in \D$ and let
$\Theta :=  F_{w} \circ \Phi $, the Liv\v{s}ic characteristic function of $Z^\Phi$ (up to a unimodular constant). Let $\ga , \tau$ be phase functions for $\Phi$, and $\Theta$, respectively with spectral functions (compositional inverses) $x,t$.
It follows that $\tau = \la _w \circ \ga$ and $\ga = \la _{-w} \circ \tau$, and by Corollary \ref{charfunspec} and Equation (\ref{periodic}),
if we define the new parameter $\vartheta := \mu _w (\theta ) = \la _w (\theta ) - \theta _w; \ \theta \in [0,1)$, then
$$ \sigma \left( Z^\Phi _{[ \vartheta + \theta _w ]} \right) = \{ t (\vartheta + \theta _w + n) | \ n \in \F = \Z \cap \ran{\tau}  \}; \quad \quad t = x\circ \la _{-w},$$ where $[s] := s - \lfloor s \rfloor$, as before.

This shows that the rate of increase of $\ga = \la _{-w} \circ \tau$ provides a measure of the local density of the sampling sequences $(t (\vartheta +\theta _w +n ) )$ with respect to the new parameter $\vartheta = \mu _w (\theta ) \in [0,1)$: 
$$  \ga ' ( t (\vartheta +\theta _w +n ) ) =  \frac{\mu _{-w} ' (\vartheta +\theta _w +n )}{ t' (\vartheta + \theta _w +n )}. $$ Namely, the size of $\ga ' (x (\theta +n )) >0$ determines how quickly the phase of $\Phi (x (\theta +n))$ is rotating, and hence measures the local density of the sampling sequences with terms $x ( \theta + n) = (t \circ \la _w ) (\theta +n ) = t (\vartheta +\theta _w +n )$.

It is, therefore, natural to extend the notion of bandlimit to the time-varying setting by
defining the time-varying bandlimit of an arbitrary local sampling space $\K (\mbf{t} , \mbf{t} ' )$ to be the function $\om : \R \rightarrow (0,\infty)$:
$$ \om (t) := \pi (\mu _w \circ \tau ) ' (t) = \pi  \tau ' (t)  \frac{1 -|w| ^2}{ | e^{i2\pi \tau(t)} -w | ^2 } > 0, $$ for some fixed choice of $w \in \D$.  As described above this will be a measure of the local density of the sampling sequences.  While it is not obvious whether there is a canonical choice of $w \in \D$, we can motivate a particular choice of $w$ that recovers the classical definition of bandlimit in the case where $K (\mbf{t} , \mbf{t} ' ) = B(A)$.

\begin{defn} \label{TVbanddefn}
     Let $(\mbf{t} , \mbf{t} ' )$ be any normalized bandlimit pair (as in Remark \ref{normpair}) with corresponding $T \in \mc{S} ^R$. Let $t, \tau$ be a fixed choice of spectral and phase functions for $T$ (fixed by a choice of equal-norm deficiency vectors $\phi _\pm$). Set $g(t) := t \coth (t); \ t\geq 0$ and let $f := g^{-1} : [1, \infty ) \rightarrow [0 , \infty ) $ be the compositional inverse of $g$.  The \emph{time-varying bandlimit} is the strictly positive function $\om : \R \rightarrow (0 , \infty )$ defined by
$$ \om (t) := \pi \ga ' (t), $$ where $\ga$ is the phase function of $F _{-w} \circ  \Theta _T$, and $w \in (0,1)$ is
$$ w := \frac{ \pi \tau' (0) - f(\pi \tau ' (0) ) }{ \pi \tau ' (0) + f(\pi \tau ' (0) ) }.$$
\end{defn}

\begin{prop}
    Let $\K (\mbf{t} , \mbf{t} ' ) = B(A)$ be the local sampling space defined by $\mbf{t} = \left( \frac{n\pi}{A} \right) _{n\in \Z}$ and $\mbf{t} ' = \left( \frac{\pi}{A} \tanh (A) \right)$. The time-varying bandlimit of $\K (\mbf{t} , \mbf{t} ')$ is the classical constant bandlimit, $\om (t) = A$. For any normalized bandlimit pair $(\mbf{t} , \mbf{t} ')$ and fixed choice of equal-norm deficiency vectors, the number $w$ in Definition \ref{TVbanddefn} belongs to $[0, 1)$.
\end{prop}

This motivates the alternate notation: If $(\mbf{t}, \mbf{t} ' )$ is a normalized bandlimit pair (in the sense of Remark \ref{normpair}), $B(\om (t) ) := \K (\mbf{t} , \mbf{t} ' ) = \K (T)$, the local bandlimit space of \emph{$\om (t)$-bandlimited functions}.

\begin{proof}
By Example \ref{PW}, $\K (\mbf{t} , \mbf{t} ' ) = B(A)$ (up to a constant rescaling), and also as before $B(A) = e^{-iAz} K (e^{i2Az} )$, let $\Phi (z) := e^{i2Az}$, a meromorphic (in fact, entire) inner function. Let $\ga$ be the phase function of $\Phi$ so that $\pi \ga (t) = A t$ and $A = \pi \ga ' (t)$.

Let $w=\Phi (i) = e^{-2A}$, and let $\tau$ be the phase function of $\Theta =  F_{w} \circ \Phi $. It follows that $\tau = \la _w \circ \ga$.
and we need to prove that $w = e^{-2A} \in (0, 1)$ is equal to the number of Definition \ref{TVbanddefn}. First calculate
$$ \tau ' (t) = \frac{A}{\pi} \frac{ 1 - e^{-4A}}{|\Phi (t) - e^{-2A} | ^2 }, $$ and, in particular
$$ \tau ' (0) = \frac{A}{\pi} \coth (A). $$
Let $g(t) := t \coth (t)$, this is smooth, strictly increasing and obeys $g ' (t) > 0$ for $t >0$. If $f$ is the compositional inverse of $g$, then we obtain
$$ A = f (\pi \tau ' (0) ). $$ It follows that
$$ w = e^{-2A} =  \frac{ \pi \tau' (0) - f(\pi \tau ' (0) ) }{ \pi \tau ' (0) + f(\pi \tau ' (0) ) }.$$

In general $g(t) > t$ for $t >0$ implies that $f(t) < t$, so that replacing $\tau$ by the phase function of an arbitrary $T \in \mc{S} ^R$ in the above definition of $w$ will always yield $w \in (0,1)$. Of course, in order that the number $w$ be well-defined, one must check that 
$\pi \tau ' (0) \geq 1$ so that $\pi \tau ' (0) \in \dom{f} = \ran{g} = [1, \infty )$. It is not difficult to check that this is always the case if $(\mbf{t} , \mbf{t} ')$ is normalized as in Remark \ref{normpair} so that 
$$ \sum \frac{t_k ' (\theta )}{1+t_k (\theta ) ^2 } = \pi; \quad \quad \theta \in [0,1 ). $$ Indeed, for any $s = \lfloor s \rfloor + [s] =: k + \theta \in \R$, we have that 
\ba \frac{t'(s)}{1+t(s) ^2} & = & \frac{t' _k (\theta )}{1+t_k(\theta ) ^2} \nn \\
&\leq & \sum \frac{t' _n (\theta )}{1+t_n (\theta ) ^2} = \pi. \nn \ea Hence, 
\ba \tau' (t) (1 +t ^2) & = & \tau ' (t(s)) (1 + t(s) ^2 ) \nn \\
& = & \frac{1 + t(s) ^2}{t'(s)} \geq \frac{1}{\pi}, \nn \ea so that $\tau ' (0) \geq \pi ^{-1}$ for any normalized bandlimit pair.
\end{proof}

\section{Measure theoretic model} \label{Measmodel}

In this section we develop a third class of models for elements $\mc{S} ^R $ using measure theory and basic spectral theory.  This connection will again provide new tools for studying local bandlimit spaces.  In particular, we will compute the unitary multiplier between any local bandlimit space $\K (T) = B( \om (t))$, and the meromorphic model space $K (\Theta _T )$ in Theorem \ref{TVmultiplier}, and if $T$ is constructed from a bandlimit pair $(\mbf{t} , \mbf{t} ')$ as in Theorem \ref{FApic}, we will provide concrete formulas expressing $\Theta _T$ in terms of any of the bandlimit pairs $(\mbf{t} _\theta , \mbf{t} _\theta ' ); \ \theta \in [0,1)$ in Corollary \ref{ACinner}. This will lead to new formulas for computing any bandlimit pair $(\mbf{t} _\theta , \mbf{t} _\theta ' )$ from the knowledge of the initial pair $(\mbf{t} , \mbf{t} ' )$. Knowledge of the sampling sequences $(\mbf{t} _\theta , \mbf{t} _\theta ')$ is, of course, necessary in order to sample and reconstruct any $f \in \K (T)$ from its samples taken on these sequences using the sampling formulas of Theorem \ref{tvbspace}.

A \emph{Herglotz} (or Nevanlinna-Herglotz) function $H$ on $\C ^+$ is an analytic function with non-negative real part. A function $H$ is a Herglotz function if and only if
there is a positive Borel measure $\Ga$ on $\R$ obeying the \emph{Herglotz condition},
\be \intfty \frac{1}{1+t^2} \Ga (dt ) < \infty, \label{Hergcondform} \ee an imaginary constant $C$, and a positive constant $D$ such that
$$ H(z) = C -iDz + \intfty \frac{zt+1}{i(t-z)} \frac{1}{1+t^2} \Ga (dt), $$ see \emph{e.g.} \cite[Section 59]{Glazman}.

There is a bijection between Herglotz functions on $\C ^+$ and the closed unit ball of (non-constant elements of) $H^\infty (\C ^+)$ given by
$$ \Theta \mapsto H_\Theta := \frac{1+\Theta}{1-\Theta},$$ with compositional inverse
$$ H \mapsto \Theta _H := \frac{H-1}{H+1}. $$

Given any such measure $\Ga$, one can consider the Hilbert space $L^2 (\Ga )$ on $\R$, and the operator $M^\Ga$ of multiplication by $t$ is a densely defined self-adjoint operator in $L^2 (\Ga )$. Define
$$ \dom{T ^\Ga } := \left\{ f \in \dom{M ^\Ga } | \ \intfty f(t) \Ga (dt ) = 0 \right\}. $$
Then by \cite[Section 3.5]{AMR}
$$ T^\Ga : = M^\Ga | _{\dom{T^\Ga }} \in \mc{S} (L^2 (\Ga )), $$ is a simple symmetric linear transformation with deficiency indices $(1,1)$ which is densely defined if and only if
$$ \Ga (\R ) = + \infty. $$

The results of \cite{AMR} further imply:
\begin{thm} \label{measuremodel}
    A linear transformation $T$ belongs to $\mc{S}$ if and only if $T \simeq T^\Ga$ for some measure $\Ga$ obeying the Herglotz condition (\ref{Hergcondform}).
\end{thm}

Given a \emph{Herglotz measure} $\Ga$ (a measure obeying the Herglotz condition), let
    $$ H (z) := \intfty \frac{zt+1}{i(t-z)} \frac{1}{1+t^2} \Ga (dt).$$

    For any $w \in \C \sm \R$, define $b_w \in L^2 (\Ga)$ via
$$ b_w (t) :=  \frac{1}{t-\ov{w}}. $$ It is easy to check that
$$ \ran{T^\Ga -z} ^\perp = \ker{(T^\Ga) ^* -\ov{z}} = \bigvee b_z. $$

In this section we fix the choice of equal-norm deficiency vectors of $T^\Ga$ to be $\phi _\pm := b _{\pm i}$. With this canonical choice, the inner meromorphic Liv\v{s}ic function of $T^\Ga$ is \cite[Section 5.4]{AMR}:
\be \Theta (z) := \frac{z-i}{z+i} \frac{\ip{b_z}{b_{+i}} _\Ga}{\ip{b_z}{b_{-i}}_\Ga } = \frac{z-i}{z+i} \frac{\intfty \frac{1}{t-z} \frac{1}{t-i} \Ga (dt)}{\intfty \frac{1}{t-z} \frac{1}{t+i} \Ga (dt)}. \label{Lcharfunfix} \ee

Let $$ \Phi := \frac{H -1}{H+1}, $$ equivalently, $H = H_\Phi = \frac{ 1 + \Phi}{1-\Phi}$. If $\Phi (i) \neq 0$ then the characteristic function
$\Theta$ of $T^\Ga$ is (up to a unimodular constant) the Frostman shift, $F_{\Phi (i) \circ \Phi}$ of $\Phi$:

\begin{lemma} \label{Frostshift}
    Let $\Ga$ be a Herglotz measure on $\R$.  If  $\Phi := \frac{H-1}{H+1}$ as above, then the Liv\v{s}ic characteristic function, $\Theta$, of $T^\Ga$ fixed by the choice $\phi _\pm = b_{\mp i}$ as in Equation (\ref{Lcharfunfix}) is:
$$ \Theta  = \left(\frac{1 - \ov{\Phi (i)}}{1 - \Phi (i)} \right) F_{\Phi (i)} \circ \Phi.$$
\end{lemma}
The proof is straightforward algebra, see \emph{e.g.} \cite[Lemma 4.4]{Martin-ext}. For example direct computation shows
$$ \Phi (z) - \Phi (i) = \frac{2}{i(H(z) +1 )(H(i) +1)} (z-i) \intfty \frac{1}{t-z}\frac{1}{t-i} \Ga (dt). $$

\begin{remark} \label{normalize}
Note that $\Phi$ is a meromorphic inner function if and only if $\Theta$ is, and that $\Theta = \Phi$ if and only if $\Phi (i) =0$ which happens if and only if $H(i) =1$, if and only if the Herglotz measure $\Ga$ is \emph{normalized} so that $$ H(i) = \intfty \frac{1}{1+t^2} \Ga (dt) = \| b_{\pm i} \| ^2 = 1.$$
\end{remark}

For the remainder of the section, we assume that the positive Herglotz measure $\Ga$ is a purely discrete sum of weighted Dirac delta masses $\delta _{t_n}$, where $(t_n ) $ is a purely discrete, strictly increasing sequence with no finite accumulation point. Namely,
$$ \Ga := \sum _n w_n \delta _{t_n}, $$ where $(w_n )$ is a sequence of strictly positive weights and
$$ \delta _{t_n} (\Om ) = \left\{ \begin{array}{cc} 1  & t_n \in \Om \\ 0 & t_n \notin \Om \end{array} \right. $$ 
We can assume without loss of generality that the $w_n >0$ for all $n$, and it follows that the sequences $\mbf{t} := (t_n)$, $\mbf{w} :=(w_n)$ obey the conditions:
\bn
    \item $\mbf{t}$ is strictly increasing with no finite accumulation point.
    \item $\mbf{w} \subset (0 , \infty )$.
    \item $\mbf{t}, \mbf{w}$ are compatible in the sense that $\sum _n \frac{w_n}{1+t_n ^2 } = A < \infty$, \emph{i.e.} $\Ga$ is a  Herglotz measure.
\en
Observe that $(\mbf{t}, \mbf{w})$ is a bandlimit pair, in the sense of Definition \ref{TVpairdef}. Indeed, we will shortly prove that up to a constant, $C$, independent of $n$, $w_n =C t_n ' = t' (n), $ where $t$ is the spectral function of $T^\Ga$ fixed by the canonical choice of deficiency vectors $\phi _{\pm} = b_{\mp i}$.

In this context Theorem \ref{measuremodel} becomes:

\begin{thm} \label{Hergmod}
    A linear transformation $T$ belongs to $\mc{S} ^R$ if and only if
$T \simeq T^\Ga$, where $\Ga = \sum _n w_n \delta _{t_n}, $ is a purely atomic Herglotz measure whose atoms have no finite accumulation point. $T^\Ga$ is densely defined if and only if
$$ \sum _{n} w_n = + \infty. $$
\end{thm}

For an atomic Herglotz measure of this type, the formula (\ref{Lcharfunfix}) for the characteristic function $\Theta$ of $M _\Ga$ (fixed uniquely by canonical deficiency vector choice) becomes:
\ba  \Theta (z) & = & \frac{z-i}{z+i} \frac{\sum \frac{1}{t_n-z} \frac{1}{t_n-i} w_n }{\sum \frac{1}{t_n -z} \frac{1}{t_n +i} w_n} \nn \\
& = & \frac{\sum w_n \left( \frac{1}{t_n -z} - \frac{1}{t_n -i} \right) }{\sum w_n \left( \frac{1}{t_n -z} - \frac{1}{t_n +i} \right) }. \nn \ea
It is also easy to check that
$$ \sigma (M ^\Ga ) = \{ t_n \} = \{ t \in \R | \ \Theta (t) = 1 \}, $$ so that if $t$ is the spectral function of $T^\Ga$ fixed by the canonical choice of deficiency vectors, $M^\Ga = T^\Ga _0 = T^\Ga (1)$, and $\sigma (M ^\Ga ) = \{ t_n \} = \{ t(n) \}$.

\begin{prop} \label{normweight}
    The weights $w_n$ of the purely discrete measure $\Ga$ obey $w_n = \frac{\| \phi _+ \| ^2}{\pi} t_n ' (0) = \frac{\| \phi _+ \| ^2}{\pi} t' (n)$.
\end{prop}
    If the atomic Herglotz measure $\Ga$ is \emph{normalized} so that $\| \phi _+ \| ^2 _\Ga =1$ as in Remark \ref{normalize} then $\pi w_n = t ' (n)$. In this case, setting $\mbf{t} ' = \pi \mbf{w} = ( \pi w_n )$, the pair $(\mbf{t} , \mbf{t} ' )$ is a normalized bandlimit pair, as defined in Remark \ref{normpair}. 

\begin{proof}
    An orthonormal basis of eigenvectors for $M^\Ga _0 = M ^\Ga$ is $\{ \phi _n \}$ where
\be \phi _n (t) := w_n ^{-1/2} \delta _{t_n ,t}. \label{ebasis} \ee Expanding the deficiency vector $\phi _+ = b_{-i}$ in this basis gives
$$ \phi _+  =\sum \frac{\sqrt{w_n}}{t_n -i} \phi _n. $$ The claim now follows from Proposition \ref{defcoeff}.
\end{proof}

It follows that we can re-express the characteristic function $\Theta$ of $T^\Ga$ in terms of the bandlimit pair
$(\mbf{t} , \mbf{t} ' )$: $\mbf{t} = (t_n )$, $\mbf{t} ' = (t _n ')$ where $t_n ' := t_n ' (0) = t ' ( n)$, $t_n = t(n)$ and
$t$ is the spectral function of $M _\Ga$ fixed by the choice of deficiency vectors $\phi _\pm = b _{\mp i}$:

\ba  \Theta (z) & = & \frac{z-i}{z+i} \frac{ \sum \frac{1}{t_n -z}\frac{1}{t_n -i} t_n '} {\sum \frac{1}{t_n -z} \frac{1}{t_n +i} t_n '} \nn \\
& = & \frac{ \sum t_n ' \left( \frac{1}{t_n -z} - \frac{1}{t_n -i} \right) }{ \sum t_n ' \left( \frac{1}{t_n -z} - \frac{1}{t_n +i} \right) }. \label{TVcharfun} \ea

This yields a representation formula for meromorphic inner functions:
\begin{cor} \label{merorep}
    A function $\Theta$ on $\C ^+$ is an inner function with meromorphic extension to $\C$ (obeying $\Theta (i) = 0$) if and only if
there is a bandlimit pair of sequences $\mbf{t} = (t_n )$, $\mbf{t} ' = (t_n ')$ so that $\Theta$ is given by the formula (\ref{TVcharfun}).

If $\Theta$ is a meromorphic inner function vanishing at $i$ then there is a phase function $\tau$ for $\Theta$, $\Theta (t) = e^{i2\pi \tau (t)}$; $t \in \R$, so that $t_n = t(n)$ and
$t ' _n = t' (n)$, where $t = \tau ^{-1}$, a spectral function for $\Theta$.
\end{cor}

\subsection{The multiplier between the local bandlimit and meromorphic model spaces} \label{Kmult}

In the construction of Section \ref{FAmodel}, one begins with a bandlimit pair $(\mbf{t}, \mbf{t} ')$ and considers a self-adjoint operator $T$
on some separable Hilbert space $\H$ so that the spectrum of $T$ consists of simple eigenvalues at the points of $\mbf{t}$. Without loss of generality, assume that the bandlimit pair is normalized so that
$$ \sum \frac{t_n ' }{1 + t_n ^2} = \pi. $$ We are also free to choose, for example, $\H = L^2 (\Ga )$ and $T = T^\Ga$ where $\Ga$ is the purely atomic Herglotz measure:
$$ \Ga = \frac{1}{\pi}  \sum t_n ' \delta _{t_n}.$$ Our normalization assumption ensures that this Herglotz measure $\Ga$ is normalized as in Remark \ref{normalize}, $\| \phi _+ \| _\Ga = \| b_{-i} \| _\Ga = 1$ so that
by Proposition \ref{normweight}, $$ t_n ' = t' (n),$$ where, as before, $t$ is the spectral function of $T=T^\Ga$ fixed by the choice $\phi _\pm = b _{\mp i}$.  It follows that the Liv\v{s}ic characteristic function of $T$ is given by equation (\ref{TVcharfun}) above. By Corollary \ref{multexist}, we know that there is a unitary multiplier from the local bandlimit space $\K (T)$ onto the meromorphic model space $K (\Theta _T )$, and we can compute this multiplier by comparing the reproducing kernels of these two spaces.

\begin{thm} \label{TVmultiplier}
    The multiplier from the meromorphic model space $K (\Theta _T)$ onto the local bandlimit space $\K (T)$ is
$$ M (t) := 2\pi ( 1 - \Theta _T (t) ) ^{-1} (-1) ^{\lfloor \tau (t)  \rfloor}  \left( \frac{t' (n)}{(t-t(n)) ^2 }\right) ^{-\frac{1}{2}}. $$
\end{thm}

\begin{proof}
   By our earlier results the kernel for the sampling or local bandlimit space is
$$ K ^T (t,s) = h(t) \sum \frac{t_k '}{ (t-t_k)(s-t_k)} \ov{h (s)}, $$ where
$$ h(t) := (-1) ^{\lfloor \tau (t)  \rfloor} f(t) = (-1) ^{\lfloor \tau (t)  \rfloor} \left( \frac{t' (n)}{(t-t(n)) ^2 }\right) ^{-\frac{1}{2}}  . $$
Suppose that $t,s \notin \{ t_n \}$ so that $H (t), H(s)$ where $H$ is the Herglotz function corresponding to $\Theta = \Theta _T$, $H = \frac{1+ \Theta }{1-\Theta }$, is well-defined. Recall here
that $\Theta (t_n ) = 1$ so that $1 - \Theta (t_n) = 0$.
Also $$ H (t) = \frac{1}{\pi} \sum \frac{t_n '}{1+t_n ^2} \frac{1 + t t_n}{i (t_n -t)}. $$ Using this one can check that
$$ \sum \frac{t_n ' }{(t-t_n) (s-t_n)} = i \pi \frac{H(t) + \ov{H (s)}}{t-s} =: \pi K^\Theta (t,s), $$ so that
$$ K ^T (t,s) = \pi h(t) K ^\Theta (t,s) \ov{h(s)}. $$ The kernel function $K^\Theta (t,s)$
can be expressed in terms of the positive kernel $k^\Theta$ for the model space $K (\Theta _T )$ as
$$ \pi K^\Theta = 4 \pi ^2 (1 - \Theta (t) ) ^{-1} k^\Theta (t,s) (1- \ov{\Theta (s)} ) ^{-1}, $$
so that $$M(t) := 2 \pi  (1 -\Theta (t) ) ^{-1} h(t), $$ is the isometric multiplier of $K (\Theta _T)$ onto $\K (T)$. The full formula for all $t,s \in \R$ follows by continuity.
\end{proof}

\subsection{Aleksandrov-Clark theory} \label{ACss}

It will be useful to re-express the characteristic function $\Theta$ of $T^\Ga$ in terms of any member of the family of bandlimit pairs $(\mbf{t} _\theta , \mbf{t} _\theta ')$, $\theta \in [0,1)$. This will provide formulas for the computation of any given sampling sequence $\mbf{t} _\theta$ for a local bandlimit space $\K (T) = \K (\mbf{t}, \mbf{t}' )$ in terms of the initial bandlimit pair $(\mbf{t}, \mbf{t} ' )$. Here, as in the previous section, consider a bandlimit pair $(\mbf{t} , \mbf{t} ')$ which is normalized so that
$$ \sum \frac{t_n ' }{1+t_n ^2} = \pi, $$ and define the normalized Herglotz measure
$$ \Ga  = \frac{1}{\pi} \sum t_n ' \delta _{t_n}. $$ If $t$ is the spectral function of $T^\Ga \in \mc{S} ^R (L^2 (\Ga ) )$ uniquely fixed by the choice $\phi _{\pm} = b_{\mp i}$ it follows as before that $t_n = t(n)$ and $t' _n = t' (n)$.

In particular, we will show that the sequences $\mbf{t} _\theta ' = (t_n ' (\theta ))$ are the weights of the purely atomic Herglotz measure $\Ga _\theta$ corresponding to the Herglotz function:
$$ H ^\theta := \frac{1 +\Theta e^{-i2\pi \theta}}{1 - \Theta e^{-i2\pi \theta}}.$$ At this point it will be clear to experts that we are simply computing the Aleksandrov-Clark measures and working out Clark's theory of unitary perturbations of the restricted backward shift in the special case of meromorphic inner functions and meromorphic model spaces on the upper half-plane \cite{Clark1972,Aleks1996}. 

Given a (not necessarily normalized) bandlimit pair $(\mbf{t} , \mbf{t} ')$, let $\Ga$ be the corresponding Herglotz measure
$$\Ga := \frac{1}{\pi} \sum t_n ' \delta _{t_n}, $$ let $H$ be the corresponding Herglotz function
and $\Phi$ the meromorphic inner function so that $H = H _\Phi = \frac{1+\Phi}{1-\Phi}$,
$$ H_\Phi (z) = \frac{ 1+ \Phi (z)}{1-\Phi (z) }= \intfty \frac{zt+1}{i(t-z)} \frac{1}{1+t^2} \Ga (dt). $$ We call $\Ga _0 := \Ga$ the \emph{Herglotz measure} of $\Phi$.
Similarly one can define a one-parameter family of positive Borel measures on $\R$ naturally associated to $\Phi$. Let $\Ga _\theta$, $\theta \in [0,1)$
be the Herglotz measure associated to the inner function $\Phi e^{-i2\pi \theta}$:
$$ \frac{ 1+ \Phi (z) e^{-i2\pi \theta} }{1-\Phi (z) e^{-i2\pi \theta} } = \intfty \frac{zt+1}{i(t-z)} \frac{1}{1+t^2} \Ga _\theta (dt). $$ The measures $\Ga _\theta$
are also called the family of \emph{Aleksandrov-Clark measures} associated to $\Phi$.
Observe that if we define the re-scaled measure
$$\ga _0 := \frac{1}{\pi} \sum \frac{t_n '}{1 + t_n ^2} \delta _{t_n}, $$ that
$$ \ga _0 (\Om ) =  \ip{\phi _+}{P^\Ga _0 (\Om ) \phi _+}; \quad \quad \Om \in \mr{Bor} (\R ),$$ where $\phi _+ = b_{-i}$ is the choice of deficiency vector from equation (\ref{defchoice}), $\mr{Bor} (\R )$ denotes the Borel $\sigma$-algebra, and
$$ P ^\Ga _0 (\Om) := \sum _{t_n \in \Om } \ip{\cdot}{\phi _n}{\phi _n}. $$ Here, the $\{ \phi _n \}$ are the orthonormal eigenbasis
of $M^\Ga =T^\Ga _0$ from equation (\ref{ebasis}), $\phi _n (t) = (w_n) ^{-1/2} \delta _{t_n, t}$, so that $P ^\Ga _0$ is the projection-valued measure of the self-adjoint operator $T^\Ga _0 = M ^\Ga$.
It follows that
$$ H_\Phi (z) = \intfty \frac{zt+1}{i(t-z)} \ga _0 (dt). $$
Similarly, define the re-scaled AC measures $\ga _\theta$ for $\theta \in [0,1)$.

Let $H ^\theta := H _{\Phi e^{-i2\pi \theta}}$, $\theta \in [0,1)$. We will write $H = H^0 = H_\Phi$. Define a positive kernel function, the \emph{Herglotz kernel}, on $\C \sm \R$ by
$$ K ^\Phi (z,w) := i  \frac{H(z) +\ov{H(w)}}{z-\ov{w}}. $$ Define $\L (\Phi ) := \H (K ^\Phi )$, the \emph{Herglotz space} of $\Phi$, a RKHS on $\C \sm \R$.
Recall that we have
$$ \ker{M _\Gamma ^* - \ov{z} } = \ran{T^\Ga -z } ^\perp = \bigvee b_z (t); \quad \quad z \in \C \sm \R, $$ where
$ b_z (t) = \frac{1}{t-\ov{z}}.$
A bit of algebra shows:
\begin{lemma}
The \emph{Cauchy transform}, $\mc{C} : L^2 (\Ga ) \rightarrow \L (\Phi )$, defined by 
$$ (\mc{C} f) (z):= \ip{b_z}{f} _\Ga ; \quad \quad z \in \C \sm \R, $$ is an onto isometry.
\end{lemma}

Let $\mc{Z} _0 ^\Phi := \mc{C} M^\Ga \mc{C} ^* = \mc{C} T^\Ga _0 \mc{C} ^*$ and $\mc{Z} ^\Phi := \mc{C} T^\Ga \mc{C} ^*$. It can be shown that $\mc{Z} ^\Phi \in \mc{S} ^R (\L (\Phi ))$ acts as multiplication
by $z$ on its maximal domain in $\L (\Phi )$ \cite[Lemma 4.3]{Martin-ext},\cite{AMR}. Comparing the kernels $K^\Phi$ for $\L (\Phi )$ and $k^\Phi$ for $K (\Phi)$ shows that
$$ K^\Phi (z) = 4 \pi (1 -\Phi (z) ) ^{-1} k^\Phi (z,w) (1 - \ov{\Phi (w)} ) ^{-1}; \quad \quad z, w \in \C ^+. $$ By Lemma \ref{multiplier},
\be m (z) := \frac{1}{2\sqrt{\pi}} (1 - \Phi (z) ), \label{HergdBmult} \ee is an onto isometric multiplier:
$$ m: \bigvee _{z \in \C ^+} K^\Phi _z \rightarrow K (\Phi ).$$

\begin{lemma}
If $\Phi$ is meromorphic and inner, $\L (\Phi ) = \bigvee _{z \in \C ^+} K_z ^\Phi$ (and $L^2 (\Ga ) = \bigvee _{z \in \C ^+} b_z$) so that $m : \L (\Phi ) \rightarrow K (\Phi )$ is a unitary multiplier.
\end{lemma}
\begin{proof}
This follows from well-known facts: Since $\Phi$ is inner, it is an extreme point of the unit ball of $H^\infty$ \cite[Chapter 9]{Hoff}. Computing the Radon-Nikodym derivative of the Herglotz measure of $\Phi$ with respect to Lebesgue measure and applying the Szeg\"{o} theorem shows that
$$ \bigvee _{z \in \C ^+} b_z = L^2 (\Ga ),$$  \cite[Chapter 4]{Hoff}. Applying the Cauchy transform isometry we obtain that $\L (\Phi ) = \bigvee _{z \in \C ^+} K_z ^\Phi$. This proves that $m$
defines an isometric multiplier of $\L (\Phi )$ onto $K (\Phi )$.
\end{proof}

In particular, it follows that
\be M_\alpha (z) := \frac{1-\Phi (z)}{1-\Phi (z) \ov{\alpha}} \label{ACmult} \ee is an isometric multiplier of $\L (\Phi )$ onto $\L (\Phi \ov{\alpha} )$ for any $\alpha \in \T$.

By Subsection \ref{multiplier}, $$ M _\alpha \mc{Z} ^\Phi M_\alpha ^* = \mc{Z} ^{\Phi \ov{\alpha}}. $$
Since $M_\alpha$ intertwines $\mc{Z} ^\Phi $ and $\mc{Z} ^{\Phi \alpha}$, it follows easily from this that if $\mc{Z} ' $ is any self-adjoint extension
of $\mc{Z} ^\Phi$ that $M \mc{Z} ' M ^*$ is a self-adjoint extension of $\mc{Z} ^{\Phi \ov{\alpha}}$ (this is Lemma \ref{multext}).

\begin{prop} \label{Cperturb}
    Let $\Phi$ be a meromorphic inner function and $\alpha , \beta \in \T$. Then $M _\alpha \mc{Z} ^\Phi (\beta)  M_\alpha ^* = \mc{Z} ^{\Phi \ov{\alpha}} (\beta \gamma ( \alpha) )$ where
$$  \ga (\alpha )  := \frac{ \ov{M_\alpha (-i)} }{\ov{M_\alpha (i)}}
= \left( \frac{1 - \Phi (i)}{1- \ov{\Phi (i)} } \right) \left( \frac{1 -\ov{\Phi (i)} \alpha}{\alpha - \Phi (i)} \right) \in \T. $$ In particular $\mc{Z} ^\Phi (\beta )$
is conjugate to $\mc{Z} ^{\Phi \ov{\delta} }$ via the unitary multiplier $M_\delta$ (of Equation \ref{ACmult}), where
$$ \delta := \frac{ \beta + \Phi (i) \frac{1 -\ov{\Phi (i)}}{1 - \Phi (i)} }{\beta \ov{\Phi (i)}
+ \frac{1 -\ov{\Phi (i)}}{1 - \Phi (i)}} \in \T.$$ If $\Phi (i) = 0$ then
$\mc{Z} ^\Phi (\beta )$ is unitarily equivalent to $ \mc{Z} ^{\Phi \ov{\beta}}$ under the unitary transformation $M_\beta$.
\end{prop}
\begin{proof}
    By Subsection \ref{multiplier}, $M _\alpha \mc{Z} ^\Phi (\beta)  M_\alpha ^* = \mc{Z} ^\alpha (\ga )$ for some $\ga \in \T$, we just need to compute $\ga$.
Equal norm deficiency vectors for $M _\Ga$ are $\phi _\pm := b _{\mp i}$, and the Cauchy transform isometry maps these onto the point evaluation
vectors $K ^\Phi _{\mp i}$ in the Herglotz space.

Let $U ^\Ga (\alpha ) = b (M^\Ga (\alpha))$ denote the unitary extensions of the Cayley transform $b ( M _\Ga )$. We know that $U ^\Ga (\alpha ) \phi _+ = \alpha \phi _-$
by construction. It follows that if $U ^\Phi (\alpha )$ are the corresponding Cayley transforms of the $\mc{Z} ^\Phi (\alpha)$ that
$U ^\Phi (\alpha) K^\Phi _{-i} = \alpha K ^\Phi _i$. It follows that there is a $\ga \in \T$ so that $$ M_\alpha U ^\Phi (\beta ) M _\alpha ^* K _{-i} ^{\Phi \ov{\alpha}} = \ga K_i ^{\Phi \ov{\alpha}}, $$ and that $M_\alpha U ^\Phi (\beta ) M_\alpha ^* = U ^{\Phi \ov{\alpha}} (\ga )$. Then,
\ba  M_\alpha U ^\Phi (\beta ) M _\alpha ^* K_{-i} ^{\Phi \ov{\alpha}} & =& \ov{M _\alpha (-i)} M_\alpha U^\Phi (\beta) K _{-i} ^\Phi \nn \\
& = & \ov{M_\alpha (-i)} \beta M _\alpha K_i ^\Phi . \nn \ea
Evaluating this at a point $z$ shows that
\ba (M_\alpha U ^\Phi (\beta ) M _\alpha ^* K_{-i} ^{\Phi \ov{\alpha}} ) (z) & = &   \beta M _\alpha (z) K_i ^\Phi (z) \ov{M_\alpha (-i)} \nn \\
& = & \beta M_\alpha (z) \frac{1}{ M_{\alpha} (z) } K_i ^{\Phi \ov{\alpha}} (z) \frac{1}{\ov{M_\alpha (i)}} \ov{M_\alpha (-i) } \nn \\
& =& \beta \frac{\ov{M_\alpha (-i)}}{\ov{M_\alpha (i)}} K_i ^{\Phi \ov{\alpha}} (z). \nn \ea
This shows that
$$ \beta \ga = \beta \ga (\alpha ) = \beta \frac{\ov{M_\alpha (-i)}}{\ov{M_\alpha (i)}}. $$
Recall here that any Herglotz function $H$ is extended to an analytic function on $\C \sm \R$ using the definition $ \ov{H(\ov{z} )} := - H (z)$, and with this definition
any inner function $\Phi$ can be extended to a meromorphic function $\C \sm \R$ by $ \ov{\Phi (\ov{z} ) } := (\Phi (z) ) ^{-1}$. It follows that
\ba \ga (\alpha ) & = & \frac{\ov{M_\alpha (-i)}}{\ov{M_\alpha (i)}} \nn \\
& = & \left( \frac{1 - \ov{\Phi (-i)}}{1- \ov{\Phi (-i)} \alpha}\right) \left( \frac{1 - \ov{\Phi (i)} \alpha}{1 - \ov{\Phi (i)}} \right) \nn \\
& = &  \left( \frac{\Phi (i) -1}{1 - \ov{\Phi (i)}} \right) \left( \frac{1 -\ov{\Phi (i)} \alpha}{\Phi (i) - \alpha} \right). \nn \ea
Setting $\beta \ga (\delta ) = 1$ and solving for $\delta$ yields the second claim. The final assertion is easy to verify.
\end{proof}

\subsection{Formulas for the characteristic function}

For the remainder of this subsection we choose our bandlimit pair $(\mbf{t}, \mbf{t} ')$ to be normalized so that atomic Herglotz measure $\Ga$ is normalized and the rescaled Herglotz measure $\ga$ is a probability measure, \emph{i.e.}
$$ \sum \frac{t_n '}{1+t_n^2} = \pi, $$
$$ \Ga = \frac{1}{\pi} \sum t_n ' \delta _{t_n}, \quad \quad \mbox{and} \quad \quad \ga = \frac{1}{\pi} \sum \frac{t_n '}{1 + t_n ^2} \delta _{t_n}. $$ Then as we have seen, $t_n ' =  t_n ' (0) =t' (n)$, where the spectral function $t$ of $T^\Ga$ is fixed by the canonical choice $\phi _\pm = b_{\mp i}$. By Remark \ref{normalize}, since the Herglotz function $H = H_\Phi$ is normalized, $\Phi (i) =0$ and $\Phi = \Theta$, the Liv\v{s}ic characteristic function of $Z^\Theta$ (fixed as always by the choice $\phi _\pm = b_{\mp i}$). It follows that all of the re-scaled Aleksandrov-Clark measures $\ga _\theta$, $\theta \in [0, 1)$ corresponding to the Herglotz functions,
$$ H ^\theta := \frac{1 + \Theta e^{-i2\pi \theta}}{1- \Theta e^{-i2\pi \theta}}, $$ are probability measures.

\begin{cor} \label{ACinner}
Let $\Ga$ be the purely atomic Herglotz measure associated to a normalized bandlimit pair of sequences $(\mbf{t}, \mbf{t} ')$ and let $T = T^\Ga \in \mc{S} ^R$. Fix the family of normalized bandlimit pairs $(\mbf{t} _\theta , \mbf{t} ' _\theta )$ by the canonical choice of equal-norm deficiency vectors $\phi _\pm = b_{\pm i}$. For any $\theta \in [0, 1)$, the Herglotz (AC) measure of $\Theta _T e^{-i2\pi \theta}$ is $\Ga _\theta = \frac{1}{\pi} \sum t' _n (\theta) \delta _{t_n (\theta)}$ so that
$$ \Theta _T (z)  =  e^{i2\pi \theta} \frac{z-i}{z+i} \frac{ \sum \frac{1}{t_n (\theta ) -z} \frac{1}{t_n (\theta) -i } t_n ' (\theta) }{  \sum \frac{1}{t_n (\theta ) -z} \frac{1}{t_n (\theta) +i } t_n ' (\theta )}  = e^{i2\pi \theta} \frac{ \sum \left( \frac{1}{t_n (\theta ) -z} -  \frac{1}{t_n (\theta) -i } \right) t_n ' (\theta) }{  \sum \left( \frac{1}{t_n (\theta ) -z} - \frac{1}{t_n (\theta) +i } \right) t_n' (\theta) }, $$ where $t_n (\theta ) = t( n+\theta)$ and $t$ is the spectral
function of $T$ (fixed by the choice $\phi _\pm = b_{\mp i}$). 
\end{cor}
By taking Frostman shifts and applying Theorem \ref{Hardymodelrep}, this provides a representation formula for arbitrary meromorphic inner functions. Using the above formula, it is easy to check that $\Theta _T (t_n (\theta ) ) = e^{i2\pi \theta}$, in agreement with Theorem \ref{extspec}.
\begin{proof}
Consider the unitary multiplier $M _\alpha : \L (\Theta ) \rightarrow \L (\Theta \ov{\alpha} )$ of Equation \ref{ACmult}. Since $\Phi = \Theta$ and $\Theta (i) =0$,
\ba  M _\alpha (-i) & = & \frac{1 - \Theta (-i) }{1-\Theta (-i) \ov{\alpha} } \nn \\
& = &  \frac{\ov{\Theta (i)} -1}{\ov{\Theta (i)} -\ov{\alpha}} \quad \quad \mbox{using that $\ov{\Theta (\ov{z} )} = \Theta (z ) ^{-1}$} \nn \\
 & = & \alpha. \nn \ea Since $M _\alpha$ is a multiplier,
\ba  M _\alpha ^* K _{-i} ^{\Theta \ov{\alpha} } & = & \ov{M_\alpha (i)} K ^\Theta _{-i} \nn \\
& =& \alpha K^\Theta _{-i}. \nn \ea
Hence setting $\alpha = e^{i2\pi \theta} \in \T$ and $\ga ^\alpha := \ga _\theta$, the re-scaled Herglotz (probability) measure of $\Theta \ov{\alpha}$,
$$ \ga ^\alpha (\Om )  = \ip{b_{-i}}{ P _\alpha (\Om ) b_{-i}} _{\Ga ^\alpha}, $$ where
$ P_\alpha$ is the projection-valued measure (PVM) of the self-adjoint operator $M^{\Ga ^\alpha} = T^{\Ga ^\alpha} _0$, and $\Ga ^\alpha = \Ga _\theta$ is the Herglotz measure of $\Theta \ov{\alpha}
= \Theta e^{-i2\pi \theta}$.
Using the unitary Cauchy transforms of $L^2 (\Ga ^\alpha)$ onto the Herglotz spaces $\L (\Phi \ov{\alpha} )$,
the fact that $\Theta (i) = 0$ and Proposition \ref{Cperturb}, this is
\ba \ga ^\alpha (\Om ) & = & \ip{ K_{-i} ^{\Theta \ov{\alpha}}}{P _{\mc{Z} ^{\Theta \ov{\alpha}}} (\Om ) K_{-i} ^{\Theta \ov{\alpha}}} \nn \\
& =& \ip{ \ov{\alpha} K ^\Theta _{-i}}{\ov{\alpha} P _{\mc{Z} ^\Theta (\alpha)} (\Om ) K_{-i} ^\Theta } \nn \\
& =& \ip{b_{-i}}{P _{T^\Ga (\alpha)} (\Om ) b _{-i}} _\Ga. \nn \ea
Let $\{ \phi _n (\theta ) \}$ be an arbitrary ONB of eigenvectors to $T^\Ga _\theta = T ^\Ga (\alpha)$.
Then recall that by Proposition \ref{defcoeff} we can write
$$ b_{-i} = \phi _+ = \sum \frac{c_n (\theta)}{t_n (\theta ) -i} \phi _n (\theta ), $$ where
$$ | c_n (\theta ) | ^2 = \frac{1}{\pi} t_n ' (\theta) = \frac{1}{\pi} t ' (n +\theta). $$ In the above we have used the assumption that $\Ga$ is normalized so that $\| \phi _+ \| =1$. It follows that
$$ \ga _\theta = \frac{1}{\pi} \sum \frac{t_n ' (\theta ) }{1 + t_n (\theta ) ^2} \delta _{t_n (\theta )}, $$ are probability measures
and so
$$ \Ga _\theta = \frac{1}{\pi} \sum t_n ' (\theta ) \delta _{t_n (\theta )}, \quad \quad \mbox{and}
\quad \quad \sum \frac{t_n ' (\theta )}{1 + t_n (\theta ) ^2} = \pi.$$
The fact that $t _n ' (\theta ) = t' (n +\theta )$ follows as in Proposition \ref{normweight}, and the formula for $\Theta = \Theta _T$ follows.
\end{proof} 

\subsection{A spectral ordinary differential equation} \label{specodie}

In this subsection we differentiate $\Theta (t) = e^{i2\pi \tau (t)}$ to obtain a first-order ordinary differential equation that characterizes spectral functions of symmetric $T \in \mc{S} ^R$. Recall that the knowledge of the spectral function is equivalent to the knowledge of all the sampling sequences $\mbf{t} _\theta; \ \theta \in [0,1)$.

By Corollary \ref{ACinner}, any meromorphic inner function $\Theta$ such that $\Theta (i) = 0$ can be written as
$$ \Theta (z) = e^{i2\pi \theta} \frac{ \sum \left( \frac{1}{t_n (\theta ) -z} -  \frac{1}{t_n (\theta) -i } \right) t_n ' (\theta) }{  \sum \left( \frac{1}{t_n (\theta ) -z} - \frac{1}{t_n (\theta) +i } \right) t_n' (\theta) } \quad \quad \theta \in [0, 1), $$ where $t_n (\theta ) = t (n +\theta )$, and $t = \tau ^{-1}$ is the compositional inverse of the phase function $\tau$ where $\Theta (t) = e ^{i2\pi \tau (t)}$.

Fix any $\theta \in [0,1)$ and choose any $t \in \R$ such that $t \notin \{ t_n (\theta ) \}$. We can write
$$ e^{-i2\pi \theta} \Theta (t) = \frac{g(t) - c}{g(t) -\ov{c} }, \quad \mbox{where} \quad g(t) := \sum \frac{t_n ' (\theta)}{t_n (\theta) -t },  \ \quad c := \sum \frac{t_n ' (\theta)}{t_n (\theta) -i }.$$
Taking the derivative yields
$$ \Theta ' (t) = \frac{ \sum \frac{ t_n ' (\theta )}{(t_n (\theta) - t ) ^2}}{\sum t_n ' (\theta ) \left( \frac{1}{t_n (\theta ) - t} - \frac{1}{t_n (\theta) +i} \right) } \left( e^{i2\pi \theta} - \Theta (t) \right). $$
Using that $\Theta (t) = e^{i2\pi \tau (t)}$ yields
$$ \tau ' (t) := \frac{1}{2\pi i} \frac{ \sum \frac{ t_n ' (\theta )}{(t_n (\theta) - t ) ^2}}{\sum t_n ' (\theta ) \left( \frac{1}{t_n (\theta ) - t} - \frac{1}{t_n (\theta) +i} \right) }\left( e^{i2\pi \theta}\ov{\Theta (t)} - 1 \right), $$ for all $t$ such that $[ \tau (t) ] \neq \theta$, \emph{i.e.} $t \neq t_n (\theta)$ for any $n$. If one takes the limit of this formula as $t \rightarrow t_n (\theta )$ the right hand side simply becomes $\tau ' ( t_n (\theta ) ) =
\frac{1}{t_n ' (\theta )}$.
This expression can be simplified slightly. Since the AC measures $\Ga _\theta$ are all normalized,
$$ \sum \frac{ t_n ' (\theta)}{1 + t_n (\theta ) ^2 } = \pi; \quad \quad \theta \in [0, 1).$$
It follows that
\ba  \tau ' (t) & = & \frac{1}{2\pi i} \sum \frac{t_n ' (\theta)}{(t_n (\theta ) - t) ^2}  \left( \frac{1}{\sum t_n ' (\theta ) \left( \frac{1}{t_n (\theta ) -t} - \frac{1}{ t_n (\theta ) -i} \right)} - \frac{1}{\sum t_n ' (\theta )  \left( \frac{1}{t_n (\theta ) -t} - \frac{1}{ t_n (\theta ) +i} \right)} \right) \nn \\
& =& \frac{1}{2\pi i}  \sum \frac{t_n ' (\theta)}{(t_n (\theta ) - t) ^2}  \frac{2i \sum \frac{t' _n (\theta)}{1 + t_n (\theta )^2}}{\left| \sum t_n ' (\theta ) \left( \frac{1}{t_n (\theta ) -t} - \frac{1}{ t_n (\theta ) -i} \right) \right| ^2} \nn \\
& =&  \frac{\sum \frac{t_n ' (\theta)}{(t_n (\theta ) - t) ^2} }{\left| \sum t_n ' (\theta ) \left( \frac{1}{t_n (\theta ) -t} - \frac{1}{ t_n (\theta ) -i} \right) \right| ^2}  =  \frac{\sum \frac{t_n ' (\theta)}{(t_n (\theta ) - t) ^2} }{(1 + t^2) \left| \sum \frac{t_n ' (\theta ) }{( t_n (\theta ) -t)(t_n (\theta ) -i)}  \right| ^2} \nn \ea
Setting $t = t(s)$ where $t = \tau ^{-1}$ yields an ordinary differential equation:

\begin{cor} \label{corodie}
    A function $t =\tau ^{-1} : \R \rightarrow (a,b)$ is the spectral function associated with a meromorphic inner function $\Theta (t) = e^{i2\pi \tau (t)}$ if and only if there is a normalized bandlimit
pair $(\mbf{t} , \mbf{t} ' )$ so that $t$ solves the first order ordinary differential equation:
$$ t'(s) :=  \frac{\left| \sum t_n '  \left( \frac{1}{t_n  -t(s)} - \frac{1}{ t_n  -i} \right) \right| ^2}{ \sum \frac{t_n ' }{(t_n  - t(s)) ^2} } = (1 + t(s) ^2) \frac{\left| \sum \frac{t_n '}{(t_n  -t(s))(t_n  -i)}  \right| ^2}{ \sum \frac{t_n ' }{(t_n  - t(s)) ^2} } .$$ subject to the initial condition $t(n) =  t_n$, for any $n$ in the index set of $\mbf{t}$.
\end{cor}

\begin{proof}
    We have already verified that if $t =\tau ^{-1}$ is the spectral function corresponding to
the meromorphic inner function $\Theta (t) = e^{i2\pi \tau (t)}$, that $t$ obeys the above differential equation and initial condition. Conversely if $t$ satisfies the above equation with initial condition for a normalized bandlimit pair $( \mbf{t}, \mbf{t} ' )$ then if $\la $ is the spectral function corresponding to this bandlimit pair, we must have that $\la ' (s) = t' (s)$ for all $s \in \ran{\tau}$. The initial condition ensures that $\la = t$ so that $t = \tau ^{-1}$ is a spectral function.
\end{proof}

Alternatively, setting $\alpha =0$ in formula (\ref{form2}),
$$ \tau ' (t _m (\theta) ) ^{-1} =  t ' (m + \theta )  = \frac{\pi ^2}{\sin ^2 (\pi \theta)} \left( \sum _n \frac{t_n ' (0 )}{(t_m (\theta ) -t_n )^2 }  \right) ^{-1}. $$ This can be written as an ordinary differential equation:
$$ t' (s) = \frac{\pi ^2}{\sin ^2 (\pi [s])} \left( \sum _n \frac{t_n ' (0 )}{(t (s) -t_n )^2 }  \right) ^{-1}.$$ Combined with the previous ordinary differential equation, this shows that the spectral function satisfies the functional equation:
\be \label{funeq} 1 + t(s) ^2 = \left| \frac{\sin (\pi [s])}{\pi} \sum \frac{t_n '}{(t_n -t(s) ) (t_n -i)} \right| ^{-2}. \ee

\section{Outlook: Application to Signal Processing} \label{SPchap}

Our general strategy for applying the local bandlimit spaces $\K (\mbf{t} , \mbf{t} ' ) = \K (T) = B (\om (t) )$ of $\om (t)$-bandlimited functions to signal processing can be summarized as follows: Given a raw signal $f_{raw}$ (or a class of such signals), 
\bn
    \item Estimate the local frequency content of $f_{raw}$ by, for example, computing a windowed Fourier transform of $f_{raw}$. That is, determine, roughly speaking, where $f_{raw}$ is varying rapidly/slowly.
         
    \item Choose a bandlimit pair of sequences $(\mbf{t} , \mbf{t} ' )$ so that the local density of the sequence $\mbf{t}$ and size of terms in $\mbf{t} '$ are proportional to the local frequency content of $f_{raw}$.
        
    \item Compute a one-parameter family of bandlimit pairs $(\mbf{t} _\theta , \mbf{t} _\theta ' )$ using the formulas of this paper.
        
    \item Apply the time-varying low-pass filter (\emph{i.e.} orthogonal projection) to $f_{raw}$ to obtain a locally bandlimited signal $f \in B(\om (t) )$. 
    
    \item Record the samples of $f$ taken on any sampling sequence $\mbf{t} _\theta$, \emph{e.g.}, for storage or transmission.
    
    \item Apply the generalized sampling formulas of Theorem \ref{tvbspace} to reconstruct the approximation $f$ to $f_{raw}$. 
\en
By choosing the spaces $B (\om (t) ) = \K (\mbf{t} , \mbf{t} ' )$ tailored to match the local frequency content of a given class of signals, we expect that this will yield
a more efficient sampling and reconstruction algorithm. This assertion is supported by \cite{Hao1,Hao2}. 

\begin{eg}
The Square Kilometer Array (SKA) telescope project
is a large international radio telescope project that is scheduled to be built in Australia and South Africa starting in 2018 \cite{Hall}. The SKA will consist of a large number of small telescopes with a total collecting area of one square kilometer. Synthetic aperture methods will allow one to combine these telescopes' data to obtain one effective telescope with an aperture of more than 3000km and a correspondingly fine resolution. The required data traffic is estimated to reach on the order of petabytes $(10^{15})$ per second, so that filtering and (essentially loss-less) compression methods for the data are of great interest. A key feature of interest here is that for any pair of SKA telescopes, their apparent distance (as seen from their target of observation) determines the amount of information by which the two telescopes' image data differs and this determines the required bandwidth of the data channel needed between them. As the earth rotates, the apparent distance of two telescopes on earth changes as a known function of time, and this means that the required bandwidth of the data channel for each pair of SKA telescopes can be naturally described using a time-varying bandlimit $\om (t)$ (proportional to the time-varying apparent distance). The new methods that we have presented here are of interest for the SKA project because they will allow one to apply a time-varying low-pass filter to such streams of continuous data with the known time-dependent bandwidth. Namely, any given image data signal that two telescopes need to share will be a function of time, $f_{raw} (t)$, such that the approximate time-varying Fourier bandwidth of $f_{raw}$ will be proportional to the known time-varying apparent distance.  This means that by choosing a time-varying bandlimit pair $(\mbf{t} , \mbf{t} ')$ according to the known local frequency content of $f_{raw}$, the image, $f = P_{\om (t) } f_{raw}  \in B(\om (t) ) = \K (\mbf{t} , \mbf{t} ' )$ under the time-varying low-pass filter $P _{\om (t)}$ (orthogonal projection onto $B (\om (t) )$), will be a good approximation to the original signal. The new methods, therefore, offer two advantages. One advantage is that in this way all noise above the time-varying bandwidth can be filtered (projected) out. This is in contrast to conventional constant-bandwidth low-pass filtering which allows one only to filter out the noise above the highest ever-occurring Fourier frequency in the signal. The other advantage is that the filtered signals $P_{\om (t)} f_{raw}$ can be reconstructed perfectly from any sampling sequence for $B (\om (t) )$, and these sequences have density proportional to $\om (t)$. Since $\om (t)$ has been chosen to match the local frequency content of the raw signals, this should reduce the sample rate while still allowing perfect and stable reconstruction. The performance and computational cost of the new methods are currently being explored with a collaborator at the SKA project, Dr. R. Dodson of the International Centre for Radio Astronomy Research (ICRAR) at the University of Western Australia.
\end{eg} 

\subsection{Future Research} 

In classical Shannon sampling theory, the speed of reconstruction of an $A-$bandlimited signal from its samples can be greatly increased by adding smooth tails to the Fourier transform of the sampling kernel.  This technique is called \emph{oversampling}. Using the equivalence of local bandlimit spaces and meromorphic model spaces, and exploiting the fact that any meromorphic model space $K (\Theta )$ can be embedding in a larger one, \emph{e.g.}, $K (\Theta ) \subset K (\Theta e^{i2Az})$, we expect that oversampling methods can be extended to the time-varying setting.  This will be an interesting direction of future research. 

The Paley-Wiener spaces $B(A)$ can also be viewed as spectral subspaces (ranges of spectral projections) of the Sturm-Liouville operator $H = -\frac{d^2}{dx^2}$. In \cite{Groch2015}, an alternate definition of time-varying bandlimit, and of locally bandlimited functions is proposed using spectral subspaces of more general Sturm-Liouville operators. It will be interesting to fully determine the connection between these two theories.

\paragraph{\bf Acknowledgements: \rm } RTWM and AK acknowledge support from the National Research Foundation of South Africa, NRF CPRR grant number $90551$. AK acknowledges support through the Discovery Programme of the National Science and Engineering Research Council of Canada (NSERC).


\end{document}